\documentclass[a4paper,12pt,reqno]{amsart}
\usepackage{latexsym}
\usepackage{color,dsfont}
\usepackage[usenames,dvipsnames]{xcolor}
\usepackage{amssymb,amsfonts,amsmath,mathrsfs}
\newcommand{\bee}{B_2^{\textrm{e}\textrm{e}}(\Phi,g)}
\newcommand{\beo}{B_2^{\text{e}\text{o}}(\Phi,g)}
\newcommand{\boe}{B_2^{\text{o}\text{e}}(\Phi,g)}
\newcommand{\boo}{B_2^{\text{o}\text{o}}(\Phi,g)}

\newtheorem{theorem}{Theorem}[section]

\newtheorem{lemma}{Lemma}[section]
\newtheorem{corollary}{Corollary}[section]
\newtheorem{conjecture}{Conjecture}[section]
\newtheorem{remark}{Remark}[section]

\begin{document}

\title{Zeros of quadratic Dirichlet $L$-functions in the hyperelliptic ensemble}
\author{H. M. Bui and Alexandra Florea}
\address{School of Mathematics, University of Manchester, Manchester M13 9PL, UK}
\email{hung.bui@manchester.ac.uk}
\address{Department of Mathematics, Stanford University, Stanford CA 94305, USA}
\email{amusat@stanford.edu}

\allowdisplaybreaks

\maketitle

\begin{abstract}
We study the $1$-level density and the pair correlation of zeros of quadratic Dirichlet $L$-functions in function fields, as we average over the ensemble $\mathcal{H}_{2g+1}$ of monic, square-free polynomials with coefficients in $\mathbb{F}_q[x]$. In the case of the $1$-level density, when the Fourier transform of the test function is supported in the restricted interval $(\frac{1}{3},1)$, we compute a secondary term of size $q^{-\frac{4g}{3}}/g$, which is not predicted by the Ratios Conjecture. Moreover, when the support is even more restricted, we obtain several lower order terms. For example, if the Fourier transform is supported in $(\frac{1}{3}, \frac{1}{2})$, we identify another lower order term of size $q^{-\frac{8g}{5}}/g$. We also compute the pair correlation, and as for the $1$-level density, we detect lower order terms under certain restrictions; for example, we see a term of size $q^{-g}/g^2$ when the Fourier transform is supported in $(\frac{1}{4},\frac{1}{2})$. The $1$-level density and the pair correlation allow us to obtain non-vanishing results for $L(\frac12,\chi_D)$, as well as lower bounds for the proportion of simple zeros of this family of $L$-functions.
\end{abstract}

\section{Introduction}
In this paper, we study the $1$-level density and the pair correlation of zeros of quadratic Dirichlet $L$-functions in function fields, which allow us to obtain non-vanishing results and lower bounds for the proportion of simple zeros. Moreover, for suitably restricted test functions, we compute several lower order terms, undetected by the powerful Ratios Conjecture \cite{ratios}.

In \cite{montgomery}, assuming the Riemann Hypothesis, Montgomery analyzed the pair correlation of the Riemann zeta-function and showed that for test functions whose Fourier transform is suitably restricted, it is equal to the pair correlation of eigenvalues of random matrices from the Gaussian Unitary Ensemble (GUE). Since Montgomery's paper, the problem of statistics of zeros in various families of $L$-functions has been well-studied. Rudnick and Sarnak \cite{rudnicksarnak}, \cite{rudnicksarnak2} assumed the Generalized Riemann Hypothesis (GRH) and showed that the $n$-level correlations of zeros of $\zeta(s)$ and $L(s,\pi)$ for primitive automorphic functions, agree with the corresponding GUE statistics, again for test functions suitably restricted.

One can study low-lying zeros of families of $L$-functions through the $n$-level densities. Katz and Sarnak \cite{katzsarnak}, \cite{katzsarnak2} conjectured that for each family of $L$-functions, there is an associated symmetry group, and the behavior of the normalized zeros near the central point agrees with the distribution of eigenvalues near $1$ in the symmetry group. There is now a wealth of literature in which the before-mentioned agreement between statistics of zeros and the distribution of eigenvalues in the classical compact groups has been proven, under certain restrictions and assumptions \cite{ILS}, \cite{rub},  \cite{hughesrudnick}, \cite{miller3}, \cite{gao}, \cite{young1}, \cite{young2}, \cite{hughesmiller}, \cite{err}.

While the Katz and Sarnak conjectures predict the main term for the $1$-level density, which is independent of the considered family of $L$-functions, we are also interested in the lower order terms, which depend on the arithmetic of the family. Relying on the Hardy-Littlewood conjecture, Bogomolny and Keating \cite{bg} proposed a more precise formula for the pair correlation function for $\zeta(s)$, which contains arithmetical information not detected by random matrix models. Furthermore, using the Ratios Conjecture of Conrey, Farmer and Zirnbauer \cite{ratios}, one can compute the $1$-level density for different families of $L$-functions, including lower order terms, specific to each family. This powerful tool allows us to write down predicted answers for various statistics of zeros, including all terms down to $X^{-\frac{1}{2}+\varepsilon}$; for example, one can compute all the $n$-level correlations of zeros of $\zeta(s)$ \cite{conreysnaith2}. For a more detailed discussion of applications of the Ratios Conjecture, we refer the readers to \cite{conreysnaith}. Lower order terms have been isolated in several families of $L$-functions under certain conditions \cite{millersymplectic}, \cite{millerorthogonal}, \cite{hks}, \cite{millerorthogonal2}, and shown to agree with the terms predicted by the Ratios Conjecture. 

In the unitary case, Fiorilli and Miller \cite{fiorillimiller} studied the $1$-level density of zeros of Dirichlet $L$-functions, for all characters modulo $q$ with $\frac Q2<q \leq Q$. For test functions with Fourier transform supported in $(-\frac{3}{2},\frac{3}{2})$, they, assuming GRH, discovered a new lower order term of size $Q^{-\frac{1}{2}}/\log Q$, undetected by the Ratios Conjecture. This is the first family in which a lower order term undetected by the Ratios Conjecture has been isolated. Our function field calculation of the $1$-level density of zeros of quadratic Dirichlet $L$-functions has the same flavor; we find a term of size 
$q^{-\frac{4g}{3}}/g$ (which translates to a term of size $ X^{-\frac{2}{3}} / \log X $ in the number field setting) when the support of the Fourier transform of the test function is sufficiently restricted. This term is smaller than the error $O_\varepsilon(X^{-\frac{1}{2} +\varepsilon})$ predicted by the Ratios Conjecture. In their paper \cite{fiorillimiller}, Fiorilli and Miller remarked that the error $O\big(Q^{-\frac{1}{2}+o(1)}\big)$ is best possible in the family of all Dirichlet $L$-functions when the support is large enough (beyond $1$), and wondered if that applies to other families. Our results do not shed light on this question; however, when the support of the Fourier transform is below $1$, the error coming from the Ratios Conjecture is not the best possible.

The $1$-level density of the family $L(s,\chi_d)$ was computed by \"Ozl\"uk and Snyder \cite{ozluksnyder} for test functions with Fourier transform supported in $(-2,2)$ on GRH, and their computations show a discontinuity in the density function at $1$. As a corollary, they proved that if GRH holds, then at least $\frac{15}{16}$ of $L(\frac12,\chi_d)$, for $d$ a fundamental discriminant with $|d| \leq X$, do not vanish. Soundararajan \cite{sound} later obtained a slightly weaker proportion, $\frac78$, without relying on GRH by studying the mollified moments of $L(s,\chi_d)$. 

Using the Ratios Conjecture, Conrey and Snaith \cite{conreysnaith} obtained an asymptotic formula for the $1$-level density in this symplectic family of $L$-functions, including lower order terms, with an error of square-root size of the main term. Miller \cite{millersymplectic} verified the Ratios Conjecture prediction for the $1$-level density for test functions with Fourier transform supported in $(-\frac13,\frac13)$, and showed that there is perfect agreement up to $X^{-\frac{1}{2}+\varepsilon}$. We also note the paper of Fiorilli, Parks and S{\"o}dergren \cite{fiorilli2}, which, under GRH, gives an asymptotic expansion in descending powers of $\log X$ when the support is in $(-2,2)$.


Here we work in $\mathbb{F}_q[x]$ with $q$ odd. Denote by $\mathcal{H}_{2g+1}$ the ensemble of monic, square-free polynomials of degree $2g+1$ over $\mathbb{F}_q[x]$. Let $\phi(\theta) = \sum_{|n| \leq N} \widehat{\phi}(n) e(n \theta)$ (with $e(x) = e^{2 \pi i x}$) be a real, even trigonometric polynomial, and let $\Phi(2g \theta) = \phi(\theta)$. Rudnick \cite{rudnicktraces} obtained an asymptotic formula for the $1$-level density 
$$ \frac{1}{ |  \mathcal{H}_{2g+1}|} \sum_{D \in \mathcal{H}_{2g+1}} \sum_{j=1}^{2g} \Phi(2g \theta_{j,D})$$ as $q$ is fixed and $g \to \infty$, where $\theta_{j,D}$ are defined in \eqref{3}, by computing the mean values of traces of powers of the Frobenius, averaged over the ensemble $\mathcal{H}_{2g+1}$. He showed that when $\widehat{\Phi}$ is supported in $(-2,2)$, the above is equal to
$$ \widehat{\Phi}(0) - \frac{1}{g} \sum_{n \leq g} \widehat{\Phi} \Big(  \frac{n}{g} \Big)  + \frac{ \text{dev}(\Phi)}{g} + o \big( 1/g \big),$$ where
$$\text{dev}(\Phi) = \widehat{\Phi}(0) \sum_{P\in\mathcal{P}} \frac{ d(P)}{|P|^2-1} - \frac{\widehat{\Phi}(1)}{q-1},$$
with the sum above being over monic irreducible polynomials $P$ of degree $d(P)$. 

We remark that for Rudnick's result, $q$ is fixed and $g \to \infty$. The other limit, of $q \to \infty$, can be treated by applying the equidistribution results of Katz and Sarnak \cite{katzsarnak}. Using basic facts about $L$-functions, $L(s,\chi_D)$ can be written as the characteristic polynomial of a certain unitary symplectic matrix. As $q \to \infty$, the conjugacy classes of these matrices become equidistributed in the group $\text{USp}(2g)$, so computing any eigenvalue statistics can be done by calculating the corresponding matrix statistics.

In our paper, we compute the $1$-level density in the ensemble $\mathcal{H}_{2g+1}$, uniformly in $q$ and $g$, and identify some lower order terms, when the support of the test function $\Phi$ is sufficiently restricted. Our results extend the result of Rudnick mentioned above. We also compute the pair correlation of zeros in this family of $L$-functions, and as for the $1$-level density, isolate several lower order terms. The paper is organized as follows. We discuss our main results in the next section. In Section \ref{preliminaries}, we provide some background information on $L$-functions over function fields, and prove several results we will use throughout the paper. Then in Section \ref{1density}, we compute the $1$-level density, and as a corollary, obtain the proportion of non-vanishing of $L(\frac12,\chi_D)$ in Section \ref{nonvan}. The pair correlation  is studied in Section \ref{2correlation}, and we prove the corollary for simple zeros in Section \ref{simplezeros}. Lastly, in the appendix, we deduce from the Ratios Conjecture the first lower order terms in the $1$-level density, and match them to the ones we already calculated. Note that a similar computation was carried out in \cite{andradekeatingconj}, but we decided to include it here, for the sake of completeness.

\section{Statement of results}

Throughout the paper we will take $q$ to be odd, with $q=p^k$. The letter $P$ will always denote a monic irreducible polynomial over $\mathbb{F}_q[x]$. The set of monic irreducible polynomials is denoted by $\mathcal{P}$, and the sets of those of degree $n$ and degree at most $n$ are denoted by $\mathcal{P}_n$ and $\mathcal{P}_{\leq n}$, respectively. For a polynomial $f\in \mathbb{F}_q[x]$, we will denote its degree by $d(f)$, and its norm $|f|$ is defined to be $q^{d(f)}$. Note that
$$|\mathcal{H}_d| = 
\begin{cases}
q & \mbox{ if } d=1, \\
q^{d-1}(q-1) & \mbox{ if } d \geq 2.
\end{cases}
$$
For any function $F$ on $\mathcal{H}_{2g+1}$, the expected value of $F$ is defined by
$$ \big\langle F \big\rangle_{\mathcal{H}_{2g+1}} := \frac{1}{| \mathcal{H}_{2g+1} |} \sum_{D \in \mathcal{H}_{2g+1}} F(D).$$

We now present the results of the paper. The following theorem provides an asymptotic formula for the $1$-level density.
\label{mainres}
\begin{theorem} 
\label{1level}
Let  $\phi(\theta) = \sum_{|n| \leq N} \widehat{\phi}(n) e(n \theta)$ be a real, even trigonometric polynomial, $\Phi(2g \theta) = \phi(\theta)$, and
\[
\Sigma_1(\Phi,g,D)=\sum_{j=1}^{2g}\Phi(2g\theta_{j,D}).
\]
If $\frac{2g}{2K+1}\leq N\leq\frac{2g-1}{2K-1}$ for some $K\geq1$, then we have
\begin{align}
&\Big\langle \Sigma_1(\Phi,g,D)\Big\rangle_{\mathcal{H}_{2g+1}}  =\widehat{\Phi}(0)-\frac{1}{g} \sum_{n\leq\frac{N}{2}}\widehat{\Phi}\Big(\frac{n}{g}\Big)+c(\Phi,g)-\sum_{k= K}^{K'}\frac{\mathds{1}_{(2k+1)|g}}{g(q-1)}\widehat{\Phi}\Big(\frac{1}{2k+1}\Big) q^{-\frac{4kg}{2k+1}} \nonumber \\
&\quad+\frac{1}{g}\sum_{k=K}^{K'}\sum_{\frac{g+1}{2k+1}\leq n\leq \min\{\frac{N}{2},\frac{g-1}{2k}\}} \widehat{\Phi}\Big(\frac{n}{g}\Big) q^{-4kn}+O \big( q^{\min\{\frac{N}{2},\frac{g-1}{2K}\}-2g-\frac12}/K \big)+O\big(q^{\frac{g-1}{K'+1}-2g}/g\big) \label{density11}
\end{align}
for any $K\leq K'\leq g$, where
\[
c(\Phi,g)=\frac{1}{g}\sum_{n\leq \frac{N}{2}}\widehat{\Phi}\Big(\frac{n}{g}\Big)\sum_{\substack{P\in\mathcal{P}_{n/r}\\r\geq1}}\frac{d(P)}{|P|^{r}(|P|+1)}.
\]
In the case $2g\leq N<4g$ we have
\begin{align*}
\Big\langle \Sigma_1&(\Phi,g,D)\Big\rangle_{\mathcal{H}_{2g+1}}  =\widehat{\Phi}(0)-\frac{1}{g} \sum_{n\leq g}\widehat{\Phi}\Big(\frac{n}{g}\Big)+c(\Phi,g)-\frac{\widehat{\Phi}(1)}{g(q-1)}+O \big( q^{\frac N2-2g-\frac12} \big).
\end{align*}
\end{theorem}

\begin{remark}
\emph{For example, when $g \leq N \leq 2g-1$ we obtain
\begin{align*}
\Big\langle \Sigma_1(\Phi,g,D)\Big\rangle_{\mathcal{H}_{2g+1}}  =&\widehat{\Phi}(0)-\frac{1}{g} \sum_{n\leq\frac N2}\widehat{\Phi}\Big(\frac{n}{g}\Big)+c(\Phi,g) -\frac{\mathds{1}_{3|g}}{g(q-1)} \widehat{\Phi}\Big(\frac{1}{3}\Big)q^{-\frac{4g}{3}} \\
&\qquad\qquad+\frac{1}{g}\sum_{\frac{g+1}{3}\leq n\leq \frac{g-1}{2}} \widehat{\Phi}\Big(\frac{n}{g}\Big) q^{-4n} +O \big(q^{-\frac{3g}{2}-1} \big).
\end{align*}
When $\frac{2g}{3}\leq N\leq g-1$ we have
\begin{align*}
&\Big\langle \Sigma_1(\Phi,g,D)\Big\rangle_{\mathcal{H}_{2g+1}}  =\widehat{\Phi}(0)-\frac{1}{g} \sum_{n\leq\frac N2}\widehat{\Phi}\Big(\frac{n}{g}\Big)+c(\Phi,g)\\
&\qquad\quad-\frac{\mathds{1}_{3|g}}{g(q-1)} \widehat{\Phi}\Big(\frac{1}{3}\Big)q^{-\frac{4g}{3}}-\frac{\mathds{1}_{5|g}}{g(q-1)}\widehat{\Phi}\Big(\frac{1}{5}\Big) q^{-\frac{8g}{5}}\\
&\qquad\qquad\qquad\quad+\frac{1}{g}\sum_{\frac{g+1}{3}\leq n\leq \frac N2} \widehat{\Phi}\Big(\frac{n}{g}\Big) q^{-4n}+\frac{1}{g}\sum_{\frac{g+1}{5}\leq n\leq \frac{g-1}{4}} \widehat{\Phi}\Big(\frac{n}{g}\Big) q^{-8n}+O \big( q^{\frac N2-2g-\frac12} \big).
\end{align*}
Hence, for example, in the restricted range $\frac{2g}{3} \leq N < \frac{4g}{5}  - 3 \log_q g$, we get a formula with two secondary main terms.}

\emph{With $\frac{2g}{5}\leq N\leq \frac{g-1}{2}$ we get
\begin{align*}
&\Big\langle\Sigma_1(\Phi,g,D)\Big\rangle_{\mathcal{H}_{2g+1}}  =\widehat{\Phi}(0)-\frac{1}{g} \sum_{n\leq\frac N2}\widehat{\Phi}\Big(\frac{n}{g}\Big)+c(\Phi,g)-\frac{\mathds{1}_{5|g}}{g(q-1)} \widehat{\Phi}\Big(\frac{1}{5}\Big)q^{-\frac{8g}{5}}\\
&\qquad\quad-\frac{\mathds{1}_{7|g}}{g(q-1)}\widehat{\Phi}\Big(\frac{1}{7}\Big) q^{-\frac{12g}{7}}-\frac{\mathds{1}_{9|g}}{g(q-1)}\widehat{\Phi}\Big(\frac{1}{9}\Big) q^{-\frac{16g}{9}}+\frac{1}{g}\sum_{\frac{g+1}{5}\leq n\leq \frac N2} \widehat{\Phi}\Big(\frac{n}{g}\Big) q^{-8n}\\
&\qquad\qquad\quad\quad+\frac{1}{g}\sum_{\frac{g+1}{7}\leq n\leq \frac{g-1}{6}} \widehat{\Phi}\Big(\frac{n}{g}\Big) q^{-12n}+\frac{1}{g}\sum_{\frac{g+1}{9}\leq n\leq \frac{g-1}{8}} \widehat{\Phi}\Big(\frac{n}{g}\Big) q^{-16n}+O \big(q^{\frac N2-2g-\frac12} \big).
\end{align*}
So, again in the restricted range $\frac{2g}{5} \leq N \leq \frac{4g}{9}- 3 \log_ q g$, we see three secondary main terms. More generally, in the $1$-level density formula \eqref{density11}, we see in total $K+1$ secondary main terms for the restricted range $\frac{2g}{2K+1} \leq N \leq \frac{4g}{4K+1} - 3 \log_q g$. The rest of the terms will be incorporated in the error term.}
\end{remark}

\begin{remark}\emph{Note that when $\frac{2g}{3} \leq N <2g$, there is a secondary main term of size $q^{-\frac{4g}{3}}/g$ in the formula for the $1$-level density. This lower order term should be compared to the secondary term computed for the first moment of the family of quadratic Dirichlet $L$-functions in the functions field setting in \cite{aflorea}. The moment conjecture of Conrey, Farmer, Keating, Rubinstein and Snaith \cite{cfkrs} in the number field case (and the corresponding conjecture in function fields \cite{andradekeatingconj}) predicts that the mean value of this family of $L$-functions is asymptotic to $ P_1(\log X)$, where $P_1$ is a linear polynomial, with an error term bounded by $O_\varepsilon(X^{-\frac{1}{2}+\varepsilon})$. In the hyperelliptic ensemble, a new term of size $g q^{-\frac{4g}{3}}$ (which translates to $X^{-\frac 23}\log X$) was computed by Florea \cite{aflorea}, smaller than the conjectured error term. This is similar to what happens for the $1$-level density: the predicted error from the Ratios Conjecture is of size $O_\varepsilon(X^{-\frac{1}{2}+\varepsilon})$, but we can compute a term of size $X^{-\frac 23}/\log X$ under certain restrictions. In fact, both these secondary terms that appear in the first moment calculation and the density, arise in a similar way, after using the Poisson summation formula in function fields and evaluating the contribution from $V$ a square polynomial, where $V$ is the dual parameter in the Poisson formula.}
\end{remark}

From Theorem \ref{1level}, we obtain the following result on the non-vanishing of $L(\frac12,\chi_D)$.
\begin{corollary}
\label{nonvanishing}
We have
\[
\frac{1}{|\mathcal{H}_{2g+1}|}\Big|\big\{D\in\mathcal{H}_{2g+1}:L(\tfrac{1}{2},\chi_D)\ne0\big\}\Big|\geq \frac{19-\cot(\frac{1}{4})}{16}+o(1)=0.9427\ldots+o(1)
\]
as $g\rightarrow\infty$.
\end{corollary}
We remark here that we can, using the usual test function $\Phi(x)=(\frac{\sin 2\pi x}{2\pi x})^2$, obtain the proportion of $\frac{15}{16}$ for the non-vanishing of $L(\frac12,\chi_D)$, as $D$ ranges over elements of $\mathcal{H}_{2g+1,q}$ and $g \to \infty$, which is the same proportion obtained by  \"Ozl\"uk and Snyder \cite{ozluksnyder} in the number field setting. Using the work of Iwaniec, Luo and Sarnak \cite{ILS}, this result can be slightly improved to give the above corollary.

Next, we turn our attention to the pair correlation of zeros in the hyperelliptic ensemble. 
\begin{theorem}
\label{paircorth}
Let
\[
\Sigma_2(\Phi,g,D)=\frac{1}{2g}\sum_{1\leq j, k\leq2g}\Phi\big(2g(\theta_{j,D}-\theta_{k,D})\big).
\]
If $\frac{g}{K+1}\leq N\leq\frac{g-1}{K}$ for some $K\geq1$, then we have
\begin{align*}
&\Big\langle \Sigma_2(\Phi,g,D)\Big\rangle_{\mathcal{H}_{2g+1}}=\widehat{\Phi}(0)+ \frac{1}{2g^2} \sum_{n \leq N} \widehat{\Phi} \Big( \frac{n}{2g} \Big) n\sum_{d|n}\frac {\alpha(d)}{d}q^{\frac nd-n}\\
&\quad\quad+\frac{1}{2g^2} \sum_{n \leq \frac N2} \widehat{\Phi} \Big( \frac{n}{g} \Big)+c_1(\Phi,g)+c_2(\Phi,g)+c_3(\Phi,g)+c_4(\Phi,g)\\
&\quad\qquad\qquad+\sum_{k= K}^{K'}\frac{\mathds{1}_{(k+1)|g}(k+1)}{2g^2(q-1)}  \widehat{\Phi} \Big( \frac{1}{2k+2} \Big)q^{-\frac{2kg}{k+1}}- \sum_{k= K}^{K'}\frac{k+1}{2g^2}  \sum_{\frac {g+1}{k+1}\leq n\leq N} \widehat{\Phi} \Big( \frac{n}{2g} \Big)q^{-2kn}\\
&\qquad\quad\quad\qquad\qquad+O\big(q^{\frac{3g}{2(K+1)}-2g-1}/g^2\big)+O\big( (\log K')q^{N-2g-1} \big)+O\big(K'q^{\frac{2(g-1)}{K'+1}-2g-1}/g^2\big)
\end{align*}
for any $K\leq K'<g$, where $\alpha(d)=\prod_{p|d}(1-p)$ and
\begin{align*}
&c_1(\Phi,g)= \frac{1}{2g^2} \sum_{n \leq N} \widehat{\Phi} \Big( \frac{n}{2g} \Big) \sum_{\substack{P \in \mathcal{P}_{n/r}\\r\geq 1}}\frac{d(P)^2}{|P|^{r}(|P|+1)},\\
&c_2(\Phi,g)=-\frac{1}{2g^2} \sum_{n \leq \frac N2} \widehat{\Phi} \Big( \frac{n}{g} \Big) \sum_{\substack{P \in \mathcal{P}_{n/r}\\r\geq 1}}\frac{d(P)^2}{|P|^{2r-2}(|P|+1)^2},\\
&c_3(\Phi,g)=-\frac{1}{g^2} \sum_{n \leq \frac N2} \widehat{\Phi} \Big( \frac{n}{g} \Big)\sum_{\substack{P \in \mathcal{P}_{n/r}\\r\geq 1}}  \frac{d(P)}{|P|^{r}(|P|+1)},\\
&c_4(\Phi,g)=\frac{1}{2g^2} \sum_{n \leq \frac N2} \widehat{\Phi} \Big( \frac{n}{g} \Big)\bigg( \sum_{\substack{P \in \mathcal{P}_{n/r}\\r\geq 1}}  \frac{d(P)}{|P|^{r}(|P|+1)}\bigg)^2.
\end{align*}
In the case $g\leq N<2g$ we have
\begin{align*}
&\Big\langle \Sigma_2(\Phi,g,D)\Big\rangle_{\mathcal{H}_{2g+1}}=\widehat{\Phi}(0)+ \frac{1}{2g^2} \sum_{n \leq N} \widehat{\Phi} \Big( \frac{n}{2g} \Big) n\sum_{d|n}\frac {\alpha(d)}{d}q^{\frac nd-n}+\frac{1}{2g^2} \sum_{n \leq \frac N2} \widehat{\Phi} \Big( \frac{n}{g} \Big)\\
&\qquad\quad+c_1(\Phi,g)+c_2(\Phi,g)+c_3(\Phi,g)+c_4(\Phi,g)+\frac{ \widehat{\Phi} ( \frac{1}{2})}{2g^2(q-1)} - \frac{1}{2g^2}  \sum_{g+1\leq n\leq N} \widehat{\Phi} \Big( \frac{n}{2g} \Big)\\
&\qquad\qquad\qquad+O\big(q^{-\frac g2-\frac12}/g^2\big)+O\big(q^{N-2g-1}\big)+O\big(q^{N-2g}/g^2\big).
\end{align*}
\end{theorem}

\begin{remark}
\emph{For example, when $\frac{g}{2} \leq N \leq g-1$ we obtain
\begin{align*}
&\Big\langle \Sigma_2(\Phi,g,D)\Big\rangle_{\mathcal{H}_{2g+1}}=\widehat{\Phi}(0)+ \frac{1}{2g^2} \sum_{n \leq N} \widehat{\Phi} \Big( \frac{n}{2g} \Big) n\sum_{d|n}\frac {\alpha(d)}{d}q^{\frac nd-n}\\
&\qquad+\frac{1}{2g^2} \sum_{n \leq \frac N2} \widehat{\Phi} \Big( \frac{n}{g} \Big)+c_1(\Phi,g)+c_2(\Phi,g)+c_3(\Phi,g)+c_4(\Phi,g)\\
&\qquad\qquad +\frac{\mathds{1}_{2|g}}{g^2(q-1)} \widehat{\Phi} \Big( \frac{1}{4} \Big) q^{-g}- \frac{1}{g^2} \sum_{\frac{g+1}{2} \leq n \leq N} \widehat{\Phi} \Big( \frac{n}{2g} \Big) q^{-2n}+ O\big(q^{-\frac{5g}{4}-1}/g^2\big) +O\big(q^{N-2g-1} \big).
\end{align*}}

\emph{With $\frac{g}{6}\leq N\leq \frac{g-1}{5}$ we get
\begin{align*}
&\Big\langle \Sigma_2(\Phi,g,D)\Big\rangle_{\mathcal{H}_{2g+1}}=\widehat{\Phi}(0)+ \frac{1}{2g^2} \sum_{n \leq N} \widehat{\Phi} \Big( \frac{n}{2g} \Big) n\sum_{d|n}\frac {\alpha(d)}{d}q^{\frac nd-n}\\
&\qquad+\frac{1}{2g^2} \sum_{n \leq \frac N2} \widehat{\Phi} \Big( \frac{n}{g} \Big)+c_1(\Phi,g)+c_2(\Phi,g)+c_3(\Phi,g)+c_4(\Phi,g)\\
&\qquad\qquad+\frac{\mathds{1}_{6|g}3}{g^2(q-1)}  \widehat{\Phi} \Big( \frac{1}{12} \Big)q^{-\frac{5g}{3}}+\frac{\mathds{1}_{7|g}7}{2g^2(q-1)}  \widehat{\Phi} \Big( \frac{1}{14} \Big)q^{-\frac{12g}{7}}\\
&\qquad\qquad\qquad- \frac{3}{g^2}  \sum_{\frac {g+1}{6}\leq n\leq N} \widehat{\Phi} \Big( \frac{n}{2g} \Big)q^{-10n}- \frac{7}{2g^2} \sum_{\frac {g+1}{7}\leq n\leq N} \widehat{\Phi} \Big( \frac{n}{2g} \Big)q^{-12n}+O\big(q^{-\frac{7g}{4}-1}/g^2\big).
\end{align*}}
\emph{In general, when $\frac{g}{K+1} \leq N \leq \frac{g-1}{K}$, we see $ \lceil \frac{K+1}{3} \rceil$ secondary main terms.}
\end{remark}

\begin{remark}\emph{We note that the first secondary term we compute appears in the range $\frac{g}{2} \leq N \leq g-1$, and is of size $q^{-g}/g^2$, which in the number field setting translates to $X^{-\frac{1}{2}}/(\log X)^2$. The moment conjecture \cite{cfkrs} predicts that the $k^{\text{th}}$ moment of quadratic Dirichlet $L$-functions should be asymptotic to $P_k(\log X)$, where $P_k$ is a polynomial of degree $\frac{k(k+1)}{2}$, with an error term of size $O_\varepsilon(X^{-\frac{1}{2}+\varepsilon})$. Soundararajan \cite{sound} obtained an asymptotic formula for the second moment with an error of size $O_\varepsilon(X^{-\frac{1}{6}+\varepsilon})$, but it is believed that in this case, there should be a genuine term of size $X^{-\frac{1}{2}+o(1)}$ appearing. Note that the second moment in the hyperelliptic ensemble was computed in \cite{aflorea3} with an error of size $O_\varepsilon(q^{-g+\varepsilon g})$, which is the equivalent of $O_\varepsilon(X^{-\frac{1}{2}+\varepsilon})$ in the number field setting. Our pair correlation result shows the existence of a term of size compatible with the discussion above. This should be compared to the $1$-level density computation, where the term of size $q^{-4g/3}/g$ that we find mirrors the secondary term in the first moment asymptotic.}

\emph{For the cubic moment of this family, there are conjectures of Diaconu, Goldfeld and Hoffstein \cite{dgh}, predicting a $X^{-\frac{1}{4}}$ term in the asymptotic formula. It would be interesting to see whether this term would be in any way reflected in the $3$-level correlation in the hyperelliptic ensemble. We will return to this question in another paper.}
\end{remark}

Theorem \ref{paircorth} allows us to get the following result on the proposition of simple zeros. 
\begin{corollary} \label{sz}
We have
\[
\Bigg\langle \frac{1}{2g}\sum_{\substack{1\leq j\leq 2g\\ \theta_{j,D}\ \emph{simple}}}1\Bigg\rangle_{\mathcal{H}_{2g+1}}\geq \frac{3}{2}-\frac{\cot(\frac{1}{\sqrt{2}})}{\sqrt{2}}+o(1)=0.6725\ldots+o(1)
\]
as $g\rightarrow\infty$.
\end{corollary}
Note that the proportion of simple zeros is the same as that obtained by Montgomery \cite{montgomery} for the Riemann-zeta function. This should not come as a surprise, since by the work of Rudnick and Sarnak \cite{rudnicksarnak}, \cite{rudnicksarnak2}, the $n$-level correlations of $\zeta(s)$ and $L(s,\pi)$ (with $L(s,\pi)$ primitive automorphic function) are universal $(n\geq2)$, and given by the correlations of the eigenvalues of matrices in the GUE.

\section{Preliminaries} \label{preliminaries}
We first give some background information on $L$-functions over function fields and their connection to zeta functions of curves.

Let $\mathcal{M}$ be the set of monic polynomials in $\mathbb{F}_q[x]$, $\mathcal{M}_n$ and $\mathcal{M}_{\leq n}$ be the sets of those of degree $n$ and degree at most $n$, respectively. Let $\pi_q(n)$ denote the number of monic, irreducible polynomials of degree $n$ over $\mathbb{F}_q[x]$. Then by the Corollary to Proposition $2.1$ in \cite{rosen}, we have the following Prime Polynomial Theorem
\begin{equation}
\pi_q(n) = \frac{1}{n} \sum_{d|n} \mu(d) q^{\frac nd}.
\label{pnt}
\end{equation}
 For $f \in \mathbb{F}_q[x]$, let
$$ \Lambda(f) = 
\begin{cases}
d(P) & \mbox{ if } f=cP^k \text{ for some }c \in \mathbb{F}_q^{\times}\ \text{and}\ k\geq 1, \\
0 & \mbox{ otherwise. }
\end{cases}
$$
We can rewrite the Prime Polynomial Theorem in the form 
\begin{equation*}
\sum_{f \in \mathcal{M}_n} \Lambda(f) = q^n.
\end{equation*}

\subsection{Quadratic Dirichlet $L$-functions over function fields}

For $\textrm{Re}(s)>1$, the zeta function of $\mathbb{F}_q[x]$ is defined by
\[
\zeta_q(s):=\sum_{f\in\mathcal{M}}\frac{1}{|f|^s}=\prod_{P\in \mathcal{P}}\bigg(1-\frac{1}{|P|^s}\bigg)^{-1}.
\]
Since there are $q^n$ monic polynomials of degree $n$, we see that
\[
\zeta_q(s)=\big(1-q^{1-s}\big)^{-1}.
\]
It is sometimes convenient to make the change of variable $u=q^{-s}$, and then write $\mathcal{Z}(u)=\zeta_q(s)$, so that $\mathcal{Z}(u)=(1-qu)^{-1}$.

For $P$ a monic irreducible polynomial, the quadratic residue symbol $\big(\frac{f}{P}\big)\in\{0,\pm1\}$ is defined by
\[
\Big(\frac{f}{P}\Big)\equiv f^{\frac{|P|-1}{2}}(\textrm{mod}\ P).
\]
If $Q=P_{1}^{\alpha_1}P_{2}^{\alpha_2}\ldots P_{r}^{\alpha_r}$, then the Jacobi symbol is defined by
\[
\Big(\frac{f}{Q}\Big)=\prod_{j=1}^{r}\Big(\frac{f}{P_j}\Big)^{\alpha_j}.
\]
The Jacobi symbol satisfies the quadratic reciprocity law. That is to say if $A,B\in \mathbb{F}_q[x]$ are relatively prime, monic polynomials, then
\[
\Big(\frac{A}{B}\Big)=(-1)^{\frac{(q-1)d(A)d(B)}{2}}\Big(\frac{B}{A}\Big).
\]

For $D$ monic, we define the character 
\[
\chi_D(g)=\Big(\frac{D}{g}\Big),
\]
and consider the $L$-function attached to $\chi_D$,
\[
L(s,\chi_D):=\sum_{f\in\mathcal{M}}\frac{\chi_D(f)}{|f|^s}.
\]
With the change of variable $u=q^{-s}$ we have
\begin{equation}\label{2}
\mathcal{L}(u,\chi_D):=L(s,\chi_D)=\sum_{f\in\mathcal{M}}\chi_D(f)u^{d(f)}=\prod_{P\in \mathcal{P}}\big(1-\chi_D(P)u^{d(P)}\big)^{-1}.
\end{equation}
For $D\in\mathcal{H}_{2g+1}$, $\mathcal{L}(u,\chi_D)$ is a polynomial in $u$ of degree $2g$ and it satisfies a functional equation
\begin{equation*}
\mathcal{L}(u,\chi_D)=(qu^2)^g\mathcal{L}\Big(\frac{1}{qu},\chi_D\Big).
\end{equation*}

There is a connection between $\mathcal{L}(u,\chi_D)$ and zeta function of curves. For $D\in\mathcal{H}_{2g+1}$, the affine equation $y^2=D(x)$ defines a projective and connected hyperelliptic curve $C_D$ of genus $g$ over $\mathbb{F}_q$. The zeta function of the curve $C_D$ is defined by
\[
Z_{C_D}(u)=\exp\bigg(\sum_{j=1}^{\infty}N_j(C_D)\frac{u^j}{j}\bigg),
\]
where $N_j(C_D)$ is the number of points on $C_D$ over $\mathbb{F}_q$, including the point at infinity. Weil \cite{weil} showed that
\[
Z_{C_D}(u)=\frac{P_{C_D}(u)}{(1-u)(1-qu)},
\]
where $P_{C_D}(u)$ is a polynomial of degree $2g$. It is known that $P_{C_D}(u)=\mathcal{L}(u,\chi_D)$ (this was proved in Artin's thesis). The Riemann Hypothesis for curves over function fields was proven by Weil \cite{weil}, so all the zeros of $\mathcal{L}(u,\chi_D)$ are on the circle $|u|=q^{-\frac12}$. We express $\mathcal{L}(u,\chi_D)$ in terms of its zeros as
\begin{equation}\label{3}
\mathcal{L}(u,\chi_D)=\prod_{j=1}^{2g}\big(1-uq^{\frac12}e^{-2\pi i\theta_{j,D}}).
\end{equation}

\subsection{Preliminary lemmas}

\begin{lemma}\label{5}
For $f\in\mathcal{M}$ we have
\[
\sum_{D\in\mathcal{H}_{2g+1}}\chi_D(f)=(-1)^{\frac{(q-1)d(f)}{2}}\bigg(\sum_{C|f^\infty}\sum_{h\in\mathcal{M}_{2g+1-2d(C)}}\chi_f(h)-q\sum_{C|f^\infty}\sum_{h\in\mathcal{M}_{2g-1-2d(C)}}\chi_f(h)\bigg),
\]
where the first summations are over monic polynomials $C$ whose prime factors are among the prime factors of $f$.
\end{lemma}
\begin{proof}
See \cite{aflorea}; Lemma $2.2$. Note that in \cite{aflorea}, $q$ was taken to be a prime with $q \equiv 1 (\textrm{mod}\ 4)$. Here, $q$ is not necessarily $\equiv 1 (\textrm{mod}\ 4)$, which is accounted for by the extra factor $(-1)^{\frac{(q-1)d(f)}{2}}$ (coming from the quadratic reciprocity formula).
\end{proof}

We define the generalized Gauss sum as
\[
G(V,\chi):= \Big(   (-1)^{\frac{(q-1)d(f)}{2}}\frac{1+i}{2}+\frac{1-i}{2}  \Big) \sum_{u (\textrm{mod}\ f)}\chi(u)e\Big(\frac{uV}{f}\Big),
\]
where the exponential was defined by Hayes \cite{hayes} as follows. For $a \in  \mathbb{F}_q\big((\frac 1x)\big) $, 
$$ e(a) = e^{ \frac{2 \pi i \text{Tr}_{\mathbb{F}_q / \mathbb{F}_p} (a_1)}{p}},$$ where $a_1$ is the coefficient of $\frac 1x$ in the Laurent expansion of $a$. Note that when $q \equiv 1 \pmod 4$, the above definition of the Gauss sum coincides with the definition in \cite{aflorea}.

Let $\chi_q$ denote the quadratic character modulo $q$ and $ \tau(q) = \sum_{a  (\textrm{mod}\ q)} \chi_q(a) e(a).$ By the Hasse-Davenport relations, $\tau(q) = (-1)^{k-1} \tau(p)^k$, and $\tau(p) = \sqrt{p}$ if $p \equiv 1 (\textrm{mod}\ 4)$, and $\tau(p) = i \sqrt{p}$ if $p \equiv 3 (\textrm{mod}\ 4)$. Let 
\begin{equation}
 \epsilon(q) = 
\begin{cases}
\frac{\tau(q)}{\sqrt{q}} & \mbox{ if } q \equiv 1(\textrm{mod}\ 4), \\
-\frac{ i \tau(q)}{\sqrt{q}} & \mbox{ if } q \equiv 3(\textrm{mod}\ 4).
\end{cases}
\label{epsilonq}
\end{equation}
Note that $\epsilon(q) = \pm 1$.

The following lemma is a slight modification of Lemma $3.2$ in \cite{aflorea}. Since the proof is very similar to the proof in \cite{aflorea}, we will only sketch it.
\begin{lemma}

\begin{enumerate}
\item If $(f,h)=1$, then $G(V, \chi_{fh})= G(V, \chi_f) G(V,\chi_h)$.
\item Write $V= V_1 P^{\alpha}$ where $P \nmid V_1$.
Then 
 $$G(V , \chi_{P^j})= 
\begin{cases}
0 & \mbox{if } j \leq \alpha \text{ and } j \text{ odd,} \\
\varphi(P^j) & \mbox{if }  j \leq \alpha \text{ and } j \text{ even,} \\
-|P|^{j-1} & \mbox{if }  j= \alpha+1 \text{ and } j \text{ even,} \\
\chi_P(V_1) |P|^{j-\frac12} & \mbox{if } j = \alpha+1 \text{ and } j \text{ odd, } d(P) \text{ even}, \\
\epsilon(q) \chi_P(V_1) |P|^{j-\frac12} & \mbox{if } j = \alpha+1 \text{ and }j \text{ odd, } d(P) \text{ odd}, \\
0 & \mbox{if } j \geq 2+ \alpha .
\end{cases}$$ 
\end{enumerate} \label{computeg}
\end{lemma}
\begin{proof}
The first part follows exactly as in the proof of Lemma $3.2$ in \cite{aflorea}, upon noticing that
\begin{align*}
\Big( (-1)^{\frac{(q-1)d(fh)}{2}} \frac{1+i}{2} + \frac{1-i}{2}  \Big) \Big(  \frac{f}{h} \Big)  \Big(  \frac{h}{f} \Big)  =&\Big( (-1)^{\frac{(q-1)d(f)}{2}}\frac{1+i}{2}+ \frac{1-i}{2}  \Big) \\
& \qquad\qquad \Big( (-1)^{\frac{(q-1)d(h)}{2}}\frac{1+i}{2}+ \frac{1-i}{2}  \Big).
\end{align*}
(The above follows easily by using the quadratic reciprocity.)

For the second part of the lemma, everything is the same as in the proof in \cite{aflorea}, except for the case $j=\alpha+1$. We now assume that $j=\alpha+1$. If $j$ is even or $d(P)$ is even, then $ (-1)^{\frac{(q-1)d(P^j)}{2} }=1$, and again everything is the same as in Lemma $3.2$ in \cite{aflorea}. If $j$ is odd and $d(P)$ is odd, then exactly the same argument shows that if $q \equiv 1(\textrm{mod}\ 4)$, then 
$$G(V,\chi_{P^j}) =\frac{\tau(q)}{\sqrt{q}} \chi_P(V_1) |P|^{j-\frac{1}{2}} ,$$ and if $q \equiv 3(\textrm{mod}\ 4)$, then
$$G(V,\chi_{P^j}) = - \frac{i \tau(q)}{\sqrt{q}}\chi_P(V_1) |P|^{j-\frac{1}{2}}.$$
\end{proof}
The next lemma concerns the Poisson summation formula for Dirichlet characters. Again, the lemma here is a minor modification of Proposition $3.1$ in \cite{aflorea}, and we will only indicate the necessary changes in the proof of Proposition $3.1$.

\begin{lemma}\label{Poisson}
Let $f\in\mathcal{M}_n$. If $n$ is even then
\[
\sum_{h\in\mathcal{M}_m}\chi_f(h)=\frac{q^m}{|f|}\bigg(G(0,\chi_f)+(q-1)\sum_{V\in\mathcal{M}_{\leq n-m-2}}G(V,\chi_f)-\sum_{V\in\mathcal{M}_{n-m-1}}G(V,\chi_f)\bigg),
\]
otherwise
\[
\sum_{h\in\mathcal{M}_m}\chi_f(h)=  \epsilon(q) \frac{q^{m+\frac12}} {|f|}\sum_{V\in\mathcal{M}_{n-m-1}}G(V,\chi_f).
\]
\end{lemma}
\begin{proof}
We proceed exactly as in \cite{aflorea}. Taking into account the definition of $G(u,\chi_f)$, we have the following (which is the analog of $(3.11)$ in \cite{aflorea}):
\begin{align*}
&\Big( (-1)^{\frac{(q-1)d(f)}{2}} \frac{1+i}{2}+ \frac{1-i}{2}\Big) \sum_{ h  \in \mathcal{M}_m} \chi_f(h)   \\
&\quad = \frac{q^m}{|f|} \bigg( G(0,\chi_f) + \sum_{d(V) \leq n-m-2} G(V,\chi_f) e \Big( \frac{-Vx^m}{f} \Big)+ \sum_{d(V) =n-m-1} G(V,\chi_f) e \Big( \frac{-Vx^m}{f} \Big) \bigg).
\end{align*}
If $d(f)$ is even, then $(-1)^{\frac{(q-1)d(f)}{2}}=1$, and everything stays the same as in \cite{aflorea}. If $d(f)$ is odd, then (again, as in \cite{aflorea})
$$ \sum_{d(V) =n-m-1} G(V,\chi_f) e \Big( \frac{-Vx^m}{f} \Big)  = \overline{\tau(q)} \sum_{V \in \mathcal{M}_{n-m-1}} G(V,\chi_f).$$ Using the above two equations finishes the proof. 
\end{proof}
Combining the two lemmas above and using the fact that $\epsilon(q)^2=1$, we obtain the following summation formula for $f=P\in\mathcal{P}$.
\begin{lemma}
\label{poissonprime}
Let $P\in\mathcal{P}_n$. If $n$ is even then
$$ \sum_{h\in\mathcal{M}_m}\chi_P(h)= \frac{q^m}{|P|^{\frac{1}{2}} }   \Big((q-1)\sum_{V\in\mathcal{M}_{\leq n-m-2}} \chi_P(V)-\sum_{V\in\mathcal{M}_{n-m-1}} \chi_P(V)\bigg),$$
otherwise
$$ \sum_{h\in\mathcal{M}_m}\chi_P(h)=   \frac{q^{m+\frac12}} {|P|^{\frac{1}{2}} } \sum_{V\in\mathcal{M}_{n-m-1}} \chi_P(V).$$
\end{lemma}
The following lemmas are the equivalent of the Polya-Vinogradov inequality and the Weil bound in function fields.
\begin{lemma} 
\label{pv}
For $f\in\mathcal{M}$ not a perfect square and $m<d(f)$ we have
$$ \sum_{h \in \mathcal{M}_m} \chi_f(h) \ll |f|^{\frac12}.$$ We also have
$$ \sum_{D \in \mathcal{H}_{2g+1}} \chi_D(P) \ll q|P|^{\frac12}\qquad\emph{and}\qquad \sum_{D \in \mathcal{H}_{2g+1}} \chi_D(PQ) \ll \frac{gq\big(d(P)+d(Q)\big)}{d(P)d(Q)}|P|^{\frac12}|Q|^{\frac12}$$
for $P\ne Q\in\mathcal{P}$.
\end{lemma}
\begin{proof}
The first inequality is proven in \cite{aflorea2}; Lemma $2.5$. For the second inequality, we use Lemma \ref{5} and then
$$ \sum_{D \in \mathcal{H}_{2g+1}} \chi_D(P) = (-1)^{\frac{(q-1)d(P)}{2}}\bigg( \sum_{j=0}^{\infty} \sum_{h \in \mathcal{M}_{2g+1-2j d(P)}} \chi_P(h) - q  \sum_{j=0}^{\infty} \sum_{h \in \mathcal{M}_{2g-1-2j d(P)}} \chi_P(h)\bigg) .$$ 
We focus on the first term above, since the second is similar to deal with. In order to get something nonzero, we need
$ 0\leq 2g+1-2jd(P)  < d(P),$ so $2j = \lfloor \frac{2g+1}{d(P)} \rfloor$, and hence there is at most one value of $j$ in the first summation above such that the character sum over $h$ is nonzero. Now we use the first part of the lemma and the conclusion follows.

The third inequality can be treated in the same way. Lemma 3.1 leads to
\begin{align*}
 \sum_{D \in \mathcal{H}_{2g+1}} \chi_{D}(PQ) =& (-1)^{\frac{(q-1)d(PQ)}{2}}\bigg( \sum_{i,j=0}^{\infty} \sum_{h \in \mathcal{M}_{2g+1-2i d(P)-2jd(Q)}} \chi_{PQ}(h)\\
 &\qquad\qquad\qquad\qquad\qquad\qquad - q  \sum_{i,j=0}^{\infty} \sum_{h \in \mathcal{M}_{2g-1-2i d(P)-2jd(Q)}} \chi_{PQ}(h) \bigg).
 \end{align*} The condition for a nonzero character sum in the first term is $ 0\leq 2g+1-2id(P)-2jd(Q)  < d(P)+d(Q)$. Before applying the first part of the lemma like before we note that there are at most $g/d(P)$ choices for $i$, and once $i$ is fixed, $\ll 1+d(P)/d(Q)$ choices for $j$. 
\end{proof}
\begin{lemma}[The Weil bound]
For $V\in\mathcal{M}$ not a perfect square we have
$$ \left|  \sum_{P \in \mathcal{P}_n} \chi_V(P) \right| \ll \frac{d(V)}{n} q^{n/2}.$$
\label{sumprimes}
\end{lemma}
\begin{proof}
See \cite{rudnicktraces}; Equation $2.5$.
\end{proof}

\begin{lemma}\label{lm4}
For $f\in\mathcal{M}$ we have
\[
\frac{1}{| \mathcal{H}_{2g+1}|}\sum_{D \in \mathcal{H}_{2g+1}} \chi_D(f^{2})=\prod_{\substack{P\in\mathcal{P}\\P|f}}\bigg(1+\frac{1}{|P|}\bigg)^{-1}+O(q^{-2g-2}).
\]
\end{lemma}
\begin{proof}
We shall prove the lemma for the case $f=P^r$. The argument can be easily extended to cover all $f\in\mathcal{M}$.

Assume $2g+1=kd(P)+a$ with $0\leq a\leq d(P)-1$. We have 
$$\sum_{D \in \mathcal{H}_{2g+1}} \chi_D(P^{2r})=\sum_{\substack{D \in \mathcal{H}_{2g+1} \\ (D,P)=1}} 1 = | \mathcal{H}_{2g+1}| -\sum_{\substack{D \in \mathcal{H}_{2g+1-d(P)} \\ (D,P)=1}} 1.$$ 
Furthermore,
$$  \sum_{\substack{D \in \mathcal{H}_{2g+1-d(P)} \\ (D,P)=1}}  1 =  \sum_{D \in \mathcal{H}_{2g+1-d(P)}}  1 - \sum_{\substack{D \in \mathcal{H}_{2g+1-2d(P)} \\ (D,P)=1}}  1=\frac{| \mathcal{H}_{2g+1}|}{|P|} - \sum_{\substack{D \in \mathcal{H}_{2g+1-2d(P)} \\ (D,P)=1}}  1.$$ We repeat this argument and get
\[
\sum_{D \in \mathcal{H}_{2g+1}} \chi_D(P^{2r})=| \mathcal{H}_{2g+1}|\sum_{j=0}^{k-1}\frac{(-1)^j}{|P|^j}+(-1)^k\sum_{\substack{D \in \mathcal{H}_{a} \\ (D,P)=1}}  1=| \mathcal{H}_{2g+1}|\sum_{j=0}^{k}\frac{(-1)^j}{|P|^j}.
\]
Hence
\[
\frac{1}{| \mathcal{H}_{2g+1}|}\sum_{D \in \mathcal{H}_{2g+1}} \chi_D(P^{2r})=\bigg(1+\frac{1}{|P|}\bigg)^{-1}+O\big(|P|^{-k-1}\big),
\]
and we obtain the lemma.
\end{proof}


\begin{lemma}\label{lambdasquared}
We have
\[
\sum_{f\in\mathcal{M}_n}\Lambda(f)^2=n\sum_{d|n}\frac {\alpha(d)}{d}q^{\frac nd},
\]
where $\alpha(d)=\prod_{p|d}(1-p)$.
\end{lemma}
\begin{proof}
From equation \eqref{pnt}, we get
\begin{align*}
\sum_{f\in\mathcal{M}_n}\Lambda(f)^2=\sum_{d|n}\frac{n^2}{d^2}\ \pi\Big(\frac{n}{d}\Big)=\sum_{d|n}\frac nd\sum_{r|\frac nd}\mu(r)q^{\frac {n}{dr}}=n\sum_{d|n}\frac {q^{\frac nd}}{d}\sum_{r|d}r\mu(r),
\end{align*}
and the lemma follows.
\end{proof}

\section{Computing the $1$-level density}
\label{1density}
Here, we will prove Theorem \ref{1level}. Recall that
\[
\Big\langle \Sigma_1(\Phi,g,D)\Big\rangle_{\mathcal{H}_{2g+1}}=\frac{1}{|\mathcal{H}_{2g+1}|}\sum_{D\in\mathcal{H}_{2g+1}}\sum_{j=1}^{2g}\Phi(2g\theta_{j,D}).
\]
By computing the logarithmic derivative $\mathcal{L}'/\mathcal{L}$ in two different ways, using \eqref{2} and \eqref{3}, we get
\[
-q^{\frac n2}\sum_{j=1}^{2g}e(-n\theta_{j,D})=\sum_{f\in\mathcal{M}_{n}}\Lambda(f)\chi_D(f).
\]
for $n>0$. So
\begin{equation}\label{7}
-\sum_{j=1}^{2g}e(n\theta_{j,D})=\sum_{f\in\mathcal{M}_{|n|}}\frac{\Lambda(f)\chi_D(f)}{|f|^{\frac12}},
\end{equation}
which is valid for $n$ both positive and negative.

Now $\Phi(2g\theta)=\frac{1}{2g}\sum_{|n|\leq N}\widehat{\Phi}(\frac{n}{2g})e(n\theta)$. Hence
\begin{eqnarray*}
\sum_{j=1}^{2g}\Phi(2g\theta_{j,D})&=&\widehat{\Phi}(0)+\frac{1}{2g}\sum_{0<|n|\leq N}\widehat{\Phi}\Big(\frac{n}{2g}\Big)\sum_{\theta_{j,D}}e(n\theta_{j,D})\\
&=&\widehat{\Phi}(0)-\frac{1}{g}\sum_{n\leq N}\widehat{\Phi}\Big(\frac{n}{2g}\Big)\sum_{f\in\mathcal{M}_{n}}\frac{\Lambda(f)\chi_D(f)}{|f|^{\frac12}}.
\end{eqnarray*}
Thus
\[
\Big\langle \Sigma_1(\Phi,g,D)\Big\rangle_{\mathcal{H}_{2g+1}}=\widehat{\Phi}(0)-A(\Phi,g),
\]
where
\[
A(\Phi,g)=\frac{1}{g|\mathcal{H}_{2g+1}|}\sum_{n\leq N}\widehat{\Phi}\Big(\frac{n}{2g}\Big)\sum_{f\in\mathcal{M}_{n}}\frac{\Lambda(f)}{|f|^{\frac12}}\sum_{D\in\mathcal{H}_{2g+1}}\chi_D(f).
\]

We decompose $A(\Phi,g)=A_1(\Phi,g)+A_2(\Phi,g)$, where
\begin{eqnarray}\label{605}
A_1(\Phi,g)=\frac{1}{g|\mathcal{H}_{2g+1}|}\sum_{n\leq N}\widehat{\Phi}\Big(\frac{n}{2g}\Big)\sum_{\substack{P\in\mathcal{P}_{n/r}\\r\ \textrm{odd}}}\frac{d(P)}{|P|^{\frac r2}}\sum_{D\in\mathcal{H}_{2g+1}}\chi_D(P^r)
\end{eqnarray}
and
\begin{eqnarray*}
A_2(\Phi,g)=\frac{1}{g|\mathcal{H}_{2g+1}|}\sum_{n\leq N/2}\widehat{\Phi}\Big(\frac{n}{g}\Big)\sum_{\substack{P\in\mathcal{P}_{n/r}\\r\geq1}}\frac{d(P)}{|P|^{r}}\sum_{D\in\mathcal{H}_{2g+1}}\chi_D(P^{2r}).
\end{eqnarray*}

\subsection{Evaluating $A_1(\Phi,g)$}

By the Polya-Vinogradov inequality in Lemma \ref{pv}, we have
$$ \sum_{D\in\mathcal{H}_{2g+1}}\chi_D(P^r)\ll q|P|^{\frac12}$$ for $r$ odd. Combining this with the Prime Polynomial Theorem, the contribution of the terms with $r\geq 3$ in \eqref{605} is
\begin{equation}\label{r=3}
  \ll g^{-1}q^{-2g}\sum_{n\leq N}\sum_{r\geq3}q^{-\frac{(r-3)n}{2r}}\ll q^{-2g},
  \end{equation} 
  and the contribution of those with $r=1$ is
  \begin{equation}\label{r=1}
  \ll q^{N-2g}/g.
  \end{equation}
  
  In what follows we shall consider various different ranges of $n$ and $N$: $\frac{2g}{2k+1}\leq n\leq N\leq\min\{\frac{2g-1}{2k},4g\}$ for $k\in\mathbb{N}$; and $\frac{2g}{2k}=\frac{g}{k}\leq n\leq N\leq\frac{2g-1}{2k-1}$ for $k\in\mathbb{N}_{\geq1}$. We denote, for $k\in\mathbb{N}$ and $k\in\mathbb{N}_{\geq1}$, respectively,
\begin{align*}
A_1(\Phi,g;2k+1)=&\frac{1}{g|\mathcal{H}_{2g+1}|}\sum_{\frac{2g}{2k+1}\leq n\leq N\leq\min\{\frac{2g-1}{2k},4g\}}\widehat{\Phi}\Big(\frac{n}{2g}\Big)\sum_{\substack{P\in\mathcal{P}_{n/r}\\r\ \textrm{odd}}}\frac{d(P)}{|P|^{\frac r2}}\sum_{D\in\mathcal{H}_{2g+1}}\chi_D(P^r)
\end{align*}
and
\begin{align*}
&A_1(\Phi,g;2k)=\frac{1}{g|\mathcal{H}_{2g+1}|}\sum_{\frac gk\leq n\leq N\leq\frac{2g-1}{2k-1}}\widehat{\Phi}\Big(\frac{n}{2g}\Big)\sum_{\substack{P\in\mathcal{P}_{n/r}\\r\ \textrm{odd}}}\frac{d(P)}{|P|^{\frac r2}}\sum_{D\in\mathcal{H}_{2g+1}}\chi_D(P^r).
\end{align*}

\subsubsection{The range $\frac{2g}{2k+1}\leq n\leq N\leq\min\{\frac{2g-1}{2k},4g\}$}\label{firstrange}
Using \eqref{r=3} it follows that
\[
A_1(\Phi,g;2k+1)=\frac{1}{g|\mathcal{H}_{2g+1}|}\sum_{\frac{2g}{2k+1}\leq n\leq N}\widehat{\Phi}\Big(\frac{n}{2g}\Big)\sum_{P\in\mathcal{P}_{n}}\frac{d(P)}{|P|^{\frac 12}}\sum_{D\in\mathcal{H}_{2g+1}}\chi_D(P)+O\big(q^{-2g}\big).
\]
In view of Lemma 3.1 we write
$$A_1(\Phi,g;2k+1)=A_{11}(\Phi,g;2k+1)-qA_{12}(\Phi,g;2k+1)+O(q^{-2g}),$$ where
$$ A_{11}(\Phi,g;2k+1) =\frac{1}{g|\mathcal{H}_{2g+1}|}\sum_{\frac{2g}{2k+1}\leq n\leq N}(-1)^{\frac{(q-1)n}{2} } \widehat{\Phi}\Big(\frac{n}{2g}\Big) \sum_{P\in\mathcal{P}_{n}}\frac{d(P)}{|P|^{\frac12}} \sum_{C| P^{\infty}}  \sum_{h \in \mathcal{M}_{2g+1-2 d(C)}} \chi_P(h)$$
and
$$ A_{12}(\Phi,g;2k+1) =\frac{1}{g|\mathcal{H}_{2g+1}|}\sum_{\frac{2g}{2k+1}\leq n\leq N}(-1)^{\frac{(q-1)n}{2} }\widehat{\Phi}\Big(\frac{n}{2g}\Big) \sum_{P\in\mathcal{P}_{n}}\frac{d(P)}{|P|^{\frac12}}\sum_{C| P^{\infty}}  \sum_{h \in \mathcal{M}_{2g-1-2 d(C)}} \chi_P(h).$$
Note that $C=P^j$ for some $j \geq 0$ in the equations above. Now for a given $P\in\mathcal{P}_{n}$, the character sum $\sum_{h \in \mathcal{M}_{2g\pm1-2j d(P)}} \chi_P(h)$ is nonzero only if $0 \leq 2g\pm1-2j d(P) < d(P).$ There is at most one value $j$, $j=j_0$ with $2j_0=\lfloor \frac{2g\pm1}{n} \rfloor=2k$, which satisfies this condition. So
$$ A_{11}(\Phi,g;2k+1) =\frac{1}{g|\mathcal{H}_{2g+1}|}\sum_{\frac{2g+2}{2k+1}\leq n\leq N}(-1)^{\frac{(q-1)n}{2} }\widehat{\Phi}\Big(\frac{n}{2g}\Big) \sum_{P\in\mathcal{P}_{n}}\frac{d(P)}{|P|^{\frac12}} \sum_{h \in \mathcal{M}_{2g+1-2kn}} \chi_P(h)$$
and
$$ A_{12}(\Phi,g;2k+1) =\frac{1}{g|\mathcal{H}_{2g+1}|}\sum_{\frac{2g}{2k+1}\leq n\leq N}(-1)^{\frac{(q-1)n}{2} }\widehat{\Phi}\Big(\frac{n}{2g}\Big)\sum_{P\in\mathcal{P}_{n}}\frac{d(P)}{|P|^{\frac12}}\sum_{h \in \mathcal{M}_{2g-1-2kn}} \chi_P(h).$$

Now we use Lemma \ref{poissonprime} for the sums over $h$ and get 
\begin{align*}
A_{11}(\Phi,g;&2k+1)  = \frac{q}{g(q-1)}\sum_{\substack{\frac{2g+2}{2k+1}\leq n\leq N \\ n \text{ even}}} \widehat{\Phi}\Big(\frac{n}{2g}\Big)\sum_{P\in\mathcal{P}_{n}}\frac{d(P)}{|P|^{2k+1}}\\
&\qquad\qquad\qquad\qquad  \bigg(  (q-1) \sum_{V \in \mathcal{M}_{\leq (2k+1)n-2g-3}} \chi_P(V)- \sum_{V \in \mathcal{M}_{(2k+1)n-2g-2}} \chi_P(V)\bigg) \\
&\qquad\qquad+ \frac{(-1)^{\frac{q-1}{2}}q^{\frac32}}{g(q-1)}\sum_{\substack{\frac{2g+2}{2k+1}\leq n\leq N \\ n \text{ odd}}} \widehat{\Phi}\Big(\frac{n}{2g}\Big)\sum_{P\in\mathcal{P}_{n}}\frac{d(P)}{|P|^{2k+1}} \sum_{V \in \mathcal{M}_{(2k+1)n-2g-2}} \chi_P(V).
\end{align*}
When $n$ is odd, $d(V)$ is odd, so $V$ cannot be a square. When $V \neq \square$, either $n$ odd or even, then using Lemma 3.6 for the sum over $P$, we get that the contribution from $V \neq \square$ is $O\big(q^{\frac N2 -2g-\frac32}/k\big)$. When $V=\square$, note that $\chi_P(V)=1$, since $d(V) <d(P)$. Hence the term in parenthesis (coming from $n$ even) is equal to
$$ (q-1) \sum_{V \in \mathcal{M}_{\leq (2k+1)\frac{n}{2}-g-2}} 1 - \sum_{V \in \mathcal{M}_{(2k+1)\frac n2-g-1}} 1 = -1.$$
Hence
\begin{align*}
A_{11}(\Phi,g;2k+1) &= -\frac{q}{g(q-1)}\sum_{\substack{\frac{2g+2}{2k+1}\leq n\leq N \\ n \text{ even}}} \widehat{\Phi}\Big(\frac{n}{2g}\Big)\sum_{P\in\mathcal{P}_{n}}\frac{d(P)}{|P|^{2k+1}} + O \big( q^{\frac N2 -2g-\frac32}/k \big)\\
&= -\frac{q}{g(q-1)}\sum_{\substack{\frac{2g+2}{2k+1}\leq n\leq N \\ n \text{ even}}} \widehat{\Phi}\Big(\frac{n}{2g}\Big) q^{-(2k+1)n} \sum_{d|n} \mu(d) q^{\frac{n}{d}}+O \big( q^{\frac N2 -2g-\frac32}/k \big)\\
&= -\frac{q}{g(q-1)}\sum_{\substack{\frac{2g+2}{2k+1}\leq n\leq N \\ n \text{ even}}} \widehat{\Phi}\Big(\frac{n}{2g}\Big) q^{-2kn} +O \big( q^{\frac N2 -2g-\frac32}/k \big).
\end{align*}
Here the second equality follows from the Prime Polynomial Theorem. 

Similarly we can compute $A_{12}(\Phi,g;2k+1)$,
\begin{align*}
A_{12}(\Phi,g;2k+1) &=-\frac{1}{gq(q-1)}\sum_{\substack{\frac{2g}{2k+1}\leq n\leq N \\ n \text{ even}}} \widehat{\Phi}\Big(\frac{n}{2g}\Big) q^{-2kn}+O \big( q^{\frac N2 -2g-\frac32}/k \big).
\end{align*}
Thus
\begin{align}\label{A12ndbound}
A_{1}(\Phi,g;2k+1) =&\frac{\mathds{1}_{(2k+1)|g}}{g(q-1)}\widehat{\Phi}\Big(\frac{1}{2k+1}\Big) q^{-\frac{4kg}{2k+1}}-\frac{1}{g}\sum_{\frac{g+1}{2k+1}\leq n\leq \frac N2} \widehat{\Phi}\Big(\frac{n}{g}\Big) q^{-4kn} \nonumber\\
&\qquad\qquad+O \big( q^{\frac N2 -2g-\frac12}/k \big).
\end{align}

\subsubsection{The range  $\frac{2g}{2k}=\frac{g}{k}\leq n\leq N\leq\frac{2g-1}{2k-1}$}
  As above we can write $$A_1(\Phi,g;2k)=A_{11}(\Phi,g;2k)-qA_{12}(\Phi,g;2k)+O(q^{-2g}),$$ where
$$ A_{11}(\Phi,g;2k) =\frac{1}{g|\mathcal{H}_{2g+1}|}\sum_{\frac gk\leq n\leq N}(-1)^{\frac{(q-1)n}{2}}\widehat{\Phi}\Big(\frac{n}{2g}\Big) \sum_{P\in\mathcal{P}_{n}}\frac{d(P)}{|P|^{\frac 12}}\sum_{C| P^{\infty}}  \sum_{h \in \mathcal{M}_{2g+1-2 d(C)}} \chi_P(h)$$
and
$$ A_{12}(\Phi,g;2k) =\frac{1}{g|\mathcal{H}_{2g+1}|}\sum_{\frac gk\leq n\leq N}(-1)^{\frac{(q-1)n}{2} }\widehat{\Phi}\Big(\frac{n}{2g}\Big)  \sum_{P\in\mathcal{P}_{n}}\frac{d(P)}{|P|^{\frac 12}}\sum_{C| P^{\infty}}  \sum_{h \in \mathcal{M}_{2g-1-2 d(C)}} \chi_P(h).$$
The only nonzero contribution comes from the term $C=P^{j_0}$ with $2j_0=\lfloor \frac{2g\pm1}{n} \rfloor$.
 It is then easy to verify that $A_{12}(\Phi,g;2k)=0$, and the only possible nonzero contribution to $A_{11}(\Phi,g;2k)$ comes from $n=\frac gk$. So
$$ A_1(\Phi,g;2k) =\frac{\mathds{1}_{k|g}(-1)^{\frac{g(q-1)}{2k}}}{g|\mathcal{H}_{2g+1}|}\widehat{\Phi}\Big(\frac{1}{2k}\Big)  \sum_{P\in\mathcal{P}_{g/k}}\frac{d(P)}{|P|^{\frac12}}\sum_{h \in \mathcal{M}_{1}} \chi_P(h)+O(q^{-2g}).$$
The Weil bound in Lemma 3.6 then implies that
\begin{equation}\label{A13rdbound}
A_1(\Phi,g;2k)\ll q^{-2g}.
\end{equation}

\subsubsection{Combining the ranges}

From \eqref{A12ndbound} and \eqref{A13rdbound} we get that if $\frac{2g}{2k+1}\leq n\leq N\leq\frac{2g-1}{2k-1}$ for $k\geq1$, or $2g\leq n\leq N<4g$ then
\begin{align*}
&\frac{1}{g|\mathcal{H}_{2g+1}|}\sum_{\frac{2g}{2k+1}\leq n\leq N}\widehat{\Phi}\Big(\frac{n}{2g}\Big)\sum_{\substack{P\in\mathcal{P}_{n/r}\\r\ \textrm{odd}}}\frac{d(P)}{|P|^{\frac r2}}\sum_{D\in\mathcal{H}_{2g+1}}\chi_D(P^r)\\
&\qquad\qquad=\frac{\mathds{1}_{(2k+1)|g}}{g(q-1)}\widehat{\Phi}\Big(\frac{1}{2k+1}\Big) q^{-\frac{4kg}{2k+1}} -\frac{1}{g}\sum_{\frac{g+1}{2k+1}\leq n\leq \min\{\frac N2,\frac{g-1}{2k}\}} \widehat{\Phi}\Big(\frac{n}{g}\Big) q^{-4kn}\nonumber\\
&\qquad\qquad\qquad\qquad+O \big( q^{\min\{\frac{N}{2},\frac{g-1}{2k}\}-2g-\frac12}/k \big).
\end{align*}
Thus, if $\frac{2g}{2K+1}\leq N\leq\frac{2g-1}{2K-1}$ for some $K\geq1$, or $2g\leq N<4g$ (corresponding to $K=0$), then we have
\begin{align*}
A_1(\Phi,g)&=\sum_{k\geq K}\frac{\mathds{1}_{(2k+1)|g}}{g(q-1)}\widehat{\Phi}\Big(\frac{1}{2k+1}\Big) q^{-\frac{4kg}{2k+1}}-\frac{1}{g}\sum_{k\geq K}\sum_{\frac{g+1}{2k+1}\leq n\leq \min\{\frac N2,\frac{g-1}{2k}\}} \widehat{\Phi}\Big(\frac{n}{g}\Big) q^{-4kn}\\
&\qquad\qquad+O \big( q^{\min\{\frac{N}{2},\frac{g-1}{2K}\}-2g-\frac12}/K \big).
\end{align*}
Also note from \eqref{r=1} and the previous subsection that we can truncate the above sums over $k$ at $k\leq K'$ at the cost of $O\big(q^{\frac{g-1}{K'+1}-2g}/g\big)$.

\subsection{Evaluating $A_2(\Phi,g)$}

From Lemma \ref{lm4} we have
\begin{eqnarray*}
A_2(\Phi,g)=\frac{1}{g}\sum_{n\leq \frac N2}\widehat{\Phi}\Big(\frac{n}{g}\Big)\sum_{\substack{P\in\mathcal{P}_{n/r}\\r\geq1}}\frac{d(P)}{|P|^{r}}-c(\Phi,g)+O\big(q^{-2g-2}\big),
\end{eqnarray*}
where
\[
c(\Phi,g)=\frac{1}{g}\sum_{n\leq \frac N2}\widehat{\Phi}\Big(\frac{n}{g}\Big)\sum_{\substack{P\in\mathcal{P}_{n/r}\\r\geq1}}\frac{d(P)}{|P|^{r}(|P|+1)}.
\]
The first term can be written as
\begin{align*}
\frac{1}{g} \sum_{n\leq \frac N2}\widehat{\Phi}\Big(\frac{n}{g}\Big)\sum_{f\in\mathcal{M}_n} \frac{\Lambda(f)}{|f|} =\frac{1}{g} \sum_{n\leq\frac N2}\widehat{\Phi}\Big(\frac{n}{g}\Big). 
\end{align*} 
Hence
\begin{eqnarray*}
A_2(\Phi,g)=\frac{1}{g} \sum_{n\leq \frac N2}\widehat{\Phi}\Big(\frac{n}{g}\Big)-c(\Phi,g)+O\big(q^{-2g-2}\big).
\end{eqnarray*}

\section{Computing the pair correlation}
\label{2correlation}
Here, we prove Theorem \ref{paircorth}. Recall that the pair correlation function is given by
\begin{align*}
\Big\langle \Sigma_2(\Phi,g,D)\Big\rangle_{\mathcal{H}_{2g+1}}=\frac{1}{2g|\mathcal{H}_{2g+1}|}\sum_{D\in\mathcal{H}_{2g+1}}\sum_{1\leq j, k\leq2g}\Phi\big(2g(\theta_{j,D}-\theta_{k,D})\big).
\end{align*}

In view of \eqref{7} we have
\[
\sum_{1\leq j,k\leq 2g}e\big(n(\theta_{j,D}-\theta_{k,D})\big)=\bigg|\sum_{f\in\mathcal{M}_{|n|}}\frac{\Lambda(f)\chi_D(f)}{|f|^{\frac12}}\bigg|^2
\]
for $n\ne0$. Hence
\begin{eqnarray*}
\frac{1}{2g}\sum_{1\leq j,k\leq 2g}\Phi\big(2g(\theta_{j,D}-\theta_{k,D})\big)&=&\widehat{\Phi}(0)+\frac{1}{4g^2}\sum_{0<|n|\leq N}\widehat{\Phi}\Big(\frac{n}{2g}\Big)\sum_{1\leq j,k\leq 2g}e\big(n(\theta_{j,D}-\theta_{k,D})\big)\\
&=&\widehat{\Phi}(0)+\frac{1}{2g^2}\sum_{n\leq N}\widehat{\Phi}\Big(\frac{n}{2g}\Big)\bigg|\sum_{f\in\mathcal{M}_{n}}\frac{\Lambda(f)\chi_D(f)}{|f|^{\frac12}}\bigg|^2.
\end{eqnarray*}
Thus
\[
\Big\langle \Sigma_2(\Phi,g,D)\Big\rangle_{\mathcal{H}_{2g+1}}=\widehat{\Phi}(0)+B(\Phi,g),
\]
where
$$ B(\Phi,g) = \frac{1}{2g^2 | \mathcal{H}_{2g+1}|} \sum_{n\leq N}\widehat{\Phi}\Big(\frac{n}{2g}\Big)\sum_{f_1,f_2 \in \mathcal{M}_n} \frac{\Lambda(f_1) \Lambda(f_2)}{|f_1|^{\frac12}|f_2|^{\frac12}} \sum_{D \in \mathcal{H}_{2g+1}} \chi_D(f_1f_2).$$ 

Let $B_1(\Phi,g)$ be the diagonal term corresponding to $f_1=f_2$ in the expression above, and let $B_2(\Phi,g)$ be the sum for which $f_1 \neq f_2$. Write $f_1=P^r$ and $f_2=Q^s.$ Further decompose $B_2(\Phi,g) = \bee+ \beo+\boe+\boo$, according to the parity of the pair $(r,s)$.

\subsection{Evaluating $B_1(\Phi,g)$}

We have
\begin{align*}
B_1(\Phi,g)=  \frac{1}{2g^2  | \mathcal{H}_{2g+1}|} \sum_{n \leq N} \widehat{\Phi} \Big( \frac{n}{2g} \Big) \sum_{\substack{P \in \mathcal{P}_{n/r}\\r\geq 1}}\frac{d(P)^2}{|P|^{r}} \sum_{D \in \mathcal{H}_{2g+1}} \chi_D(P^{2r}).
\end{align*}
From Lemma \ref{lm4} we get
\begin{align*}
B_1(\Phi,g)=  \frac{1}{2g^2  } \sum_{n \leq N} \widehat{\Phi} \Big( \frac{n}{2g} \Big) \sum_{\substack{P \in \mathcal{P}_{n/r}\\r\geq 1}}\frac{d(P)^2}{|P|^{r}}+c_1(\Phi,g)+O\big(q^{-2g-2}\big),
\end{align*}
where
\[
c_1(\Phi,g)= -\frac{1}{2g^2} \sum_{n \leq N} \widehat{\Phi} \Big( \frac{n}{2g} \Big) \sum_{\substack{P \in \mathcal{P}_{n/r}\\r\geq 1}}\frac{d(P)^2}{|P|^{r}(|P|+1)}.
\]
We rewrite the first term as
\[
 \frac{1}{2g^2  } \sum_{n \leq N} \widehat{\Phi} \Big( \frac{n}{2g} \Big) \sum_{f \in \mathcal{M}_{n}}\frac{\Lambda(f)^2}{|f|}
\]
Using Lemma \ref{lambdasquared} leads to
 $$B_{1}(\Phi,g) = \frac{1}{2g^2} \sum_{n \leq N} \widehat{\Phi} \Big( \frac{n}{2g} \Big) n\sum_{d|n}\frac {\alpha(d)}{d}q^{\frac nd-n} +c_1(\Phi)+O\big(q^{-2g-2}\big).$$

\subsection{Evaluating $\bee$}
We have
\begin{align*}
\bee =  \frac{1}{2g^2  | \mathcal{H}_{2g+1}|} \sum_{n \leq \frac N2} \widehat{\Phi} \Big( \frac{n}{g} \Big) \sum_{\substack{P \in \mathcal{P}_{n/r}\\r\geq 1}} \sum_{\substack{ Q \in \mathcal{P}_{{n/s}} \\ Q\ne P,\ s\geq 1}} \frac{d(P)d(Q)}{|P|^r|Q|^s} \sum_{D \in \mathcal{H}_{2g+1}} \chi_D(P^{2r}Q^{2s}).
\end{align*}
Hence, using Lemma 3.7,
\begin{align*}
\bee=&\frac{1}{2g^2} \sum_{n \leq \frac N2} \widehat{\Phi} \Big( \frac{n}{g} \Big) \sum_{\substack{P \in \mathcal{P}_{n/r}\\r\geq 1}} \sum_{\substack{ Q \in \mathcal{P}_{{n/s}} \\ Q\ne P,\ s\geq 1}} \frac{d(P)d(Q)}{|P|^{r-1}|Q|^{s-1}(|P|+1)(|Q|+1)}\\
&\qquad\qquad+O(q^{-2g-2}/g)\\
=&\frac{1}{2g^2} \sum_{n \leq \frac N2} \widehat{\Phi} \Big( \frac{n}{g} \Big)\bigg( \sum_{\substack{P \in \mathcal{P}_{n/r}\\r\geq 1}}  \frac{d(P)}{|P|^{r-1}(|P|+1)}\bigg)^2+c_2(\Phi,g)+O\big(q^{-2g-2}/g\big),
\end{align*}
where
\begin{align*}
c_2(\Phi,g)=-\frac{1}{2g^2} \sum_{n \leq N/2} \widehat{\Phi} \Big( \frac{n}{g} \Big) \sum_{\substack{P \in \mathcal{P}_{n/r}\\r\geq 1}}\frac{d(P)^2}{|P|^{2r-2}(|P|+1)^2}.
\end{align*}
For the first term, we write it as
\[
\frac{1}{2g^2} \sum_{n \leq \frac N2} \widehat{\Phi} \Big( \frac{n}{g} \Big)\bigg( \sum_{\substack{P \in \mathcal{P}_{n/r}\\r\geq 1}}  \frac{d(P)}{|P|^{r}}-\sum_{\substack{P \in \mathcal{P}_{n/r}\\r\geq 1}}  \frac{d(P)}{|P|^{r}(|P|+1)}\bigg)^2.
\]
From the Prime Polynomial Theorem, this is equal to
\[
\frac{1}{2g^2} \sum_{n \leq\frac N2} \widehat{\Phi} \Big( \frac{n}{g} \Big)+c_3(\Phi,g)+c_4(\Phi,g),
\]
where
\[
c_3(\Phi,g)=-\frac{1}{g^2} \sum_{n \leq\frac N2} \widehat{\Phi} \Big( \frac{n}{g} \Big)\sum_{\substack{P \in \mathcal{P}_{n/r}\\r\geq 1}}  \frac{d(P)}{|P|^{r}(|P|+1)}
\]
and
\[
c_4(\Phi,g)=\frac{1}{2g^2} \sum_{n \leq \frac N2} \widehat{\Phi} \Big( \frac{n}{g} \Big)\bigg( \sum_{\substack{P \in \mathcal{P}_{n/r}\\r\geq 1}}  \frac{d(P)}{|P|^{r}(|P|+1)}\bigg)^2.
\]
It follows that
\begin{equation*}
\bee=  \frac{1}{2g^2} \sum_{n \leq \frac N2} \widehat{\Phi} \Big( \frac{n}{g} \Big) +c_2(\Phi)+c_3(\Phi)+c_4(\Phi)+O(q^{-2g-2}/g).
\end{equation*}

\subsection{Evaluating $\beo$ and $\boe$}
We shall focus on $\beo$. Note that $\boe = \beo$. We have
\begin{align}
\beo= \frac{1}{2g^2 | \mathcal{H}_{2g+1}|} \sum_{\substack{ n \leq  N \\ n \text{ even}}} \widehat{\Phi} \Big( \frac{n}{2g} \Big) \sum_{\substack{P \in \mathcal{P}_{n/2r}\\r\geq1}} \sum_{\substack{Q \in \mathcal{P}_{n/(2s+1)}\\s\geq 0}} \frac{d(P) d(Q)}{|P|^r |Q|^{s+\frac{1}{2}}} \sum_{\substack{D \in \mathcal{H}_{2g+1} \\ (D,P)=1}} \chi_D(Q). \label{exp2}
\end{align}
Write 
$$\sum_{\substack{D \in \mathcal{H}_{2g+1} \\ (D,P)=1}} \chi_D(Q) = \sum_{D \in \mathcal{H}_{2g+1}} \chi_D(Q) - \chi_P(Q) \sum_{\substack{D \in \mathcal{H}_{2g+1-d(P)} \\ (D,P)=1}}  \chi_D(Q).$$ 
The first term above (the full sum over $D$) is bounded by $q|Q|^{\frac12}$, by using the Polya-Vinogradov inequality in function fields from Lemma \ref{pv}. We further express
$$  \sum_{\substack{D \in \mathcal{H}_{2g+1-d(P)} \\ (D,P)=1}}  \chi_{D}(Q) =  \sum_{D \in \mathcal{H}_{2g+1-d(P)}}  \chi_{D}(Q) - \chi_P(Q)  \sum_{\substack{D \in \mathcal{H}_{2g+1-2d(P)} \\ (D,P)=1}}  \chi_{D}(Q),$$ and again bound the full sum over $D$ by $q|Q|^{\frac12}$. We repeat this argument, and it follows that
$$ \sum_{\substack{D \in \mathcal{H}_{2g+1} \\ (D,P)=1}}  \chi_{D}(Q) \ll  \frac{gq}{d(P)}|Q|^{\frac12}.$$
We use this bound in \eqref{exp2} and use the Prime Polynomial Theorem twice, for the sums over $P$ and $Q$, and end up with
\begin{equation}
\beo \ll q^{N-2g}/gN. \label{beoest}
\end{equation}

\subsection{Evaluating $\boo$}\label{boosection}
We have
\begin{align*}
\boo &=  \frac{1}{2g^2 | \mathcal{H}_{2g+1}|}\sum_{n \leq N} \widehat{\Phi} \Big( \frac{n}{2g} \Big)\sum_{\substack{P \in \mathcal{P}_{n/(2r+1)}\\r\geq0}} \sum_{\substack{Q \in \mathcal{P}_{n/(2s+1)} \\ Q\ne P,\ s\geq 0}}  \frac{d(P) d(Q)}{ |P|^{r+\frac12} |Q|^{s+\frac12}} \sum_{D \in \mathcal{H}_{2g+1}} \chi_D(PQ). 
\end{align*}
By the third inequality in Lemma 3.5,
\[
\sum_{D \in \mathcal{H}_{2g+1}} \chi_D(PQ)\ll \frac{gq\big(d(P)+d(Q)\big)}{d(P)d(Q)}|P|^{\frac12}|Q|^{\frac12}.
\]
Combining with the Prime Polynomial Theorem, the contribution of the terms with $r,s\geq 1$ to $\boo$ is
\begin{equation}\label{boo101}
\ll g^{-1}q^{-2g}\sum_{n\leq N}\sum_{r,s\geq 1}q^{-\frac{(r-1)n}{2r+1}-\frac{(s-1)n}{2s+1}}\frac{r+s+1}{n}\ll q^{-2g}(\log N)/g,
\end{equation}
and that of the terms with $r=0$, $s\geq 1$ or $r\geq1$, $s=0$ is
\begin{equation}\label{boo102}
\ll q^{N-2g}/gN.
\end{equation}
For $N$ small, we also bound $\boo$ by
\begin{equation}\label{boo120}
\boo \ll q^{2N-2g}/gN.
\end{equation}

We now consider various different ranges of $n$ and $N$: $\frac{g}{k+1}\leq n\leq N\leq\min\{\frac{g-1}{k},2g\}$ for $k\in\mathbb{N}$. We denote, for $k\in\mathbb{N}$,
\begin{align*}
B_{2}^{\textrm{oo}}(\Phi,g;k) &=  \frac{1}{2g^2 | \mathcal{H}_{2g+1}|}\sum_{\frac {g}{k+1}\leq n \leq N\leq\min\{\frac{g-1}{k},2g\}} \widehat{\Phi} \Big( \frac{n}{2g} \Big)\\
&\qquad\qquad\qquad\sum_{\substack{P \in \mathcal{P}_{n/(2r+1)}\\r\geq0}} \sum_{\substack{Q \in \mathcal{P}_{n/(2s+1)} \\ Q\ne P,\ s\geq 0}}  \frac{d(P) d(Q)}{ |P|^{r+\frac12} |Q|^{s+\frac12}} \sum_{D \in \mathcal{H}_{2g+1}} \chi_D(PQ). 
\end{align*}

\subsubsection{The range $\frac{g}{k+1}\leq n\leq N\leq\min\{\frac{g-1}{k},2g\}$, $k\in\mathbb{N}
$}

Using \eqref{boo101} and \eqref{boo102} we get
\begin{align*}
B_{2}^{\textrm{oo}}(\Phi,g;k) &=  \frac{1}{2g^2 | \mathcal{H}_{2g+1}|}\sum_{\frac {g}{k+1}\leq n \leq N} \widehat{\Phi} \Big( \frac{n}{2g} \Big)\sum_{P\ne Q \in \mathcal{P}_{n}} \frac{d(P) d(Q)}{ |P|^{\frac12} |Q|^{\frac12}} \sum_{D \in \mathcal{H}_{2g+1}} \chi_D(PQ)\\
&\qquad\qquad\qquad +O\big((k+1)q^{N-2g}/g^2\big). 
\end{align*}
Applying Lemma \ref{5} (and using the fact that $d(PQ)$ is even),
\begin{align*}
 \sum_{D \in \mathcal{H}_{2g+1}} \chi_D(PQ) &= \sum_{C| (PQ)^{\infty}} \sum_{h \in \mathcal{M}_{2g+1-2d(C)}} \chi_{PQ}(h)  - q \sum_{C| (PQ)^{\infty}} \sum_{h \in \mathcal{M}_{2g-1-2d(C)}} \chi_{PQ}(h) \\
 &= \sum_{i,j=0}^{\infty} \sum_{h \in \mathcal{M}_{2g+1-2(i+j)n}} \chi_{PQ}(h) - q \sum_{i,j=0}^{\infty} \sum_{h \in \mathcal{M}_{2g-1-2(i+j)n}} \chi_{PQ}(h),
 \end{align*}
and then we can write
\[
B_{2}^{\textrm{oo}}(\Phi,g;k) =B_{21}^{\textrm{oo}}(\Phi,g;k)-qB_{22}^{\textrm{oo}}(\Phi,g;k)+O\big((k+1)q^{N-2g}/g^2\big),
\]
 where
\begin{align*}
B_{21}^{\textrm{oo}}(\Phi,g;k)=\frac{1}{2g^2 | \mathcal{H}_{2g+1}|}\sum_{\frac {g}{k+1}\leq n\leq N} \widehat{\Phi} \Big( \frac{n}{2g} \Big)\sum_{i,j=0}^{\infty}  \sum_{P\ne Q \in \mathcal{P}_{n}} \frac{d(P) d(Q)}{|P|^{\frac12}|Q|^{\frac12}} \sum_{h \in \mathcal{M}_{2g+1-2(i+j)n}} \chi_{PQ}(h)
 \end{align*}
 and $B_{22}^{\textrm{oo}}(\Phi,g;k)$ has a similar expression with the sum over $h\in \mathcal{M}_{2g+1-2(i+j)n}$ being replaced by $h\in\mathcal{M}_{2g-1-2(i+j)n}$.
Given $P$ and $Q$, the character sums $\sum_{h \in \mathcal{M}_{2g\pm1-2(i+j)n}} \chi_{PQ}(h)$ are nonzero only if 
\begin{equation}\label{ij}
0 \leq 2g\pm1-
2(i+j)n < 2n,
\end{equation} 
respectively. Hence, in this range, $i+j=k$ for both sums, or $i+j=k+1$ for the first character sum if $n=\frac{g}{k+1}$.

Next we use Lemma \ref{Poisson} for the sum over $h$, and together with Lemma \ref{computeg} and the fact that $\epsilon(q)^2=1$, we get that
\begin{align}\label{new100}
&B_{21}^{\textrm{oo}}(\Phi,g;k)= \frac{q}{2g^2(q-1)}  \sum_{\frac {g}{k+1}\leq n\leq N} \widehat{\Phi} \Big( \frac{n}{2g} \Big)\sum_{i,j}{\!}^{\flat}\  q^{-2(i+j)n}\sum_{P\ne Q \in \mathcal{P}_{n}} \frac{d(P) d(Q)}{|P||Q|} \\
&\qquad\qquad\qquad\qquad \Big( (q-1) \sum_{V \in \mathcal{M}_{\leq 2(i+j+1)n-2g-3}} \chi_{PQ}(V)- \sum_{V \in \mathcal{M}_{2(i+j+1)n-2g-2}} \chi_{PQ}(V) \Big), \nonumber
\end{align} 
where $\sum^\flat$ denotes $i+j=k$, or $i+j=k+1$ if $n=\frac{g}{k+1}$.
We then write
\[
B_{21}^{\textrm{oo}}(\Phi,g;k)=B_{21}^{\textrm{oo}}(k;V=\square)+B_{21}^{\textrm{oo}}(k;V\ne\square),
\]
where $B_{21}^{\textrm{oo}}(k;V=\square)$ corresponds to the sum over $V$ being a square polynomial and $B_{21}^{\textrm{oo}}(k;V\ne\square)$ corresponds to the sum over $V$ non-square. 

For $B_{21}^{\textrm{oo}}(k;V=\square)$, we note that $d(V)<2n$, and since $V=\square$, we have $(V,PQ)=1$. So the expression in the bracket in \eqref{new100} is
\begin{align*}
(q-1) \sum_{V \in \mathcal{M}_{\leq (i+j+1)n-g-2}} 1- \sum_{V \in \mathcal{M}_{(i+j+1)n-g-1}} 1=-1.
\end{align*}
We reintroduce the terms with $P=Q$ and obtain
\begin{align*}
B_{21}^{\textrm{oo}}(k;V=\square) & = - \frac{q}{2g^2(q-1)}  \sum_{\frac {g}{k+1}\leq n\leq N} \widehat{\Phi} \Big( \frac{n}{2g} \Big)\sum_{i,j}{\!}^{\flat}\  q^{-2(i+j)n} \sum_{P, Q \in \mathcal{P}_{n}} \frac{d(P) d(Q)}{|P||Q|}\\
&\qquad\qquad+\frac{q}{2g^2(q-1)}  \sum_{\frac {g}{k+1}\leq n\leq N} \widehat{\Phi} \Big( \frac{n}{2g} \Big)\sum_{i,j}{\!}^{\flat}\  q^{-2(i+j)n} \sum_{P \in \mathcal{P}_{n}} \frac{d(P)^2}{|P|^{2}}.
\end{align*}
In view of equation \eqref{pnt} this equals
\begin{align}\label{boo105}
&- \frac{q}{2g^2(q-1)}  \sum_{\frac {g}{k+1}\leq n\leq N} \widehat{\Phi} \Big( \frac{n}{2g} \Big)\sum_{i,j}{\!}^{\flat}\  q^{-2(i+j)n}\Big(1+O\big(q^{-\frac{n}{2}}\big)\Big)\nonumber\\
&\qquad\qquad=- \frac{(k+1)q}{2g^2(q-1)}  \sum_{\frac {g+1}{k+1}\leq n\leq N} \widehat{\Phi} \Big( \frac{n}{2g} \Big)q^{-2kn}+O\big((k+1)q^{\frac{3(g+1)}{2(k+1)}-2g-2}/g^2\big).
\end{align}

On the other hand, we have
\begin{align*}
B_{21}^{\textrm{oo}}&(k;V\ne\square) = \frac{q}{2g^2(q-1)}  \sum_{\frac{g}{k+1}\leq n\leq N} \widehat{\Phi} \Big( \frac{n}{2g} \Big)\sum_{i,j}{\!}^{\flat}\ q^{-2(i+j)n}\\
&\qquad\quad\qquad\qquad\bigg((q-1) \sum_{\substack{V \in \mathcal{M}_{\leq 2(i+j+1)n-2g-3} \\ V \neq \square}}  \sum_{P\ne Q \in \mathcal{P}_{n}} \frac{d(P) d(Q)}{|P||Q|} \chi_{PQ}(V) \\
&\qquad\qquad\qquad\qquad\qquad\qquad\qquad\quad- \sum_{\substack{V \in \mathcal{M}_{2(i+j+1)n-2g-2} \\ V \neq \square}}   \sum_{P\ne Q \in \mathcal{P}_{n}} \frac{d(P) d(Q)}{|P||Q|} \chi_{PQ}(V) \bigg).
\end{align*}
We add and then subtract the terms with $P=Q$. As for the complete expression over all $P$ and $Q$, we use the Weil bound in Lemma \ref{sumprimes} twice and then trivially bound the sum over $V$. In doing so we get
\begin{align*}
&B_{21}^{\textrm{oo}}(k;V\ne\square) =- \frac{q}{2g^2(q-1)}  \sum_{\frac{g}{k+1}\leq n\leq N} \widehat{\Phi} \Big( \frac{n}{2g} \Big)\sum_{i,j}{\!}^{\flat}\ q^{-2(i+j)n}\sum_{P \in \mathcal{P}_{n}} \frac{d(P)^2}{|P|^{2}}\\
&\qquad\qquad\qquad\bigg((q-1) \sum_{\substack{V \in \mathcal{M}_{\leq 2(i+j+1)n-2g-3} \\ V \neq \square,\ (V,P)=1}}1- \sum_{\substack{V \in \mathcal{M}_{2(i+j+1)n-2g-2} \\ V \neq \square,\ (V,P)=1}} 1\bigg)+O\big(q^{N-2g-2}/(k+1)\big).
\end{align*}
The expression in the above bracket is
\begin{align*}
&\bigg((q-1) \sum_{V \in \mathcal{M}_{\leq 2(i+j+1)n-2g-3}} 1- \sum_{V \in \mathcal{M}_{2(i+j+1)n-2g-2}} 1\bigg)\\
&\qquad\qquad-\bigg((q-1) \sum_{V \in \mathcal{M}_{\leq (2i+2j+1)n-2g-3}}  1- \sum_{V \in \mathcal{M}_{ (2i+2j+1)n-2g-2}} 1\bigg)\\
&\qquad\qquad\qquad\qquad-\bigg((q-1) \sum_{V \in \mathcal{M}_{\leq (i+j+1)n-g-2}} 1- \sum_{V \in \mathcal{M}_{(i+j+1)n-g-1}} 1\bigg)\ll 1.
\end{align*}
Hence, using the Prime Polynomial Theorem,
\[
B_{21}^{\textrm{oo}}(k;V\ne\square)\ll q^{\frac{g+1}{k+1}-2g-2}/g+q^{N-2g-2}/(k+1).
\]

Adding the estimate for $B_{21}^{\textrm{oo}}(k;V=\square)$ in \eqref{boo105} we get
\begin{align*}
B_{21}^{\textrm{oo}}(\Phi,g;k)=&- \frac{(k+1)q}{2g^2(q-1)}  \sum_{\frac {g+1}{k+1}\leq n\leq N} \widehat{\Phi} \Big( \frac{n}{2g} \Big)q^{-2kn}\\
&\qquad\qquad\qquad+O\big((k+1)q^{\frac{3(g+1)}{2(k+1)}-2g-2}/g^2\big)+O\big(q^{N-2g-2}/(k+1)\big).
\end{align*}

Similarly,
\begin{align*}
B_{22}^{\textrm{oo}}(\Phi,g;k)=&- \frac{k+1}{2g^2q(q-1)}  \sum_{\frac {g}{k+1}\leq n\leq N} \widehat{\Phi} \Big( \frac{n}{2g} \Big)q^{-2kn}\\
&\qquad\qquad\qquad+O\big((k+1)q^{\frac{3g}{2(k+1)}-2g-2}/g^2\big)+O\big(q^{N-2g-2}/(k+1)\big).
\end{align*}
Thus,
\begin{align*}
B_{2}^{\textrm{oo}}(\Phi,g;k)=&\frac{\mathds{1}_{(k+1)|g}(k+1)}{2g^2(q-1)}  \widehat{\Phi} \Big( \frac{1}{2k+2} \Big)q^{-\frac{2kg}{k+1}}- \frac{k+1}{2g^2}  \sum_{\frac {g+1}{k+1}\leq n\leq N} \widehat{\Phi} \Big( \frac{n}{2g} \Big)q^{-2kn}\\
&\qquad\quad+O\big((k+1)q^{\frac{3(g+1)}{2(k+1)}-2g-2}/g^2\big)+O\big((k+1)q^{\frac{3g}{2(k+1)}-2g-1}/g^2\big)\\
&\qquad\qquad\qquad\qquad+O\big(q^{N-2g-1}/(k+1)\big)+O\big((k+1)q^{N-2g}/g^2\big).
\end{align*}

\subsubsection{Combining the ranges} If $\frac{g}{K+1}\leq N\leq\frac{g-1}{K}$ for some $K\geq1$, then combining the estimates for $B_{2}^{\textrm{oo}}(\Phi,g;k)$ for all $k\geq K$ we have
\begin{align*}
B_{2}^{\textrm{oo}}(\Phi,g)=&\sum_{k\geq K}\frac{\mathds{1}_{(k+1)|g}(k+1)}{2g^2(q-1)}  \widehat{\Phi} \Big( \frac{1}{2k+2} \Big)q^{-\frac{2kg}{k+1}}- \sum_{k\geq K}\frac{k+1}{2g^2}  \sum_{\frac {g+1}{k+1}\leq n\leq N} \widehat{\Phi} \Big( \frac{n}{2g} \Big)q^{-2kn}\\
&\qquad\qquad+O\big((K+1)q^{\frac{3g}{2(K+1)}-2g-1}/g^2\big)+O\big((\log g)q^{N-2g-1} \big)+O\big(q^{N-2g}\big).
\end{align*}
We note from \eqref{boo120} that we can also truncate the above sums over $k$ at $k\leq K'$ and replace the last two error terms by $O\big((\log K')q^{N-2g-1} \big)+O\big(K'q^{\frac{2(g-1)}{K'+1}-2g}/g^2\big)$.

When $g\leq N<2g$ (which corresponds to $K=0$),
\begin{align*}
B_{2}^{\textrm{oo}}(\Phi,g)=&\frac{ \widehat{\Phi} ( \frac{1}{2})}{2g^2(q-1)} - \frac{1}{2g^2}  \sum_{g+1\leq n\leq N} \widehat{\Phi} \Big( \frac{n}{2g} \Big)\\
&\qquad\qquad+O\big(q^{-\frac g2-\frac12}/g^2\big)+O\big(q^{N-2g-1}\big)+O\big(q^{N-2g}/g^2\big).
\end{align*}

\section{Proof of Corollary \ref{nonvanishing}}
\label{nonvan}
We shall give a sketch of the proof of Corollary 2.1, and refer interested readers to the work of Iwaniec, Luo and Sarnak \cite{ILS}, Appendix A, for a thorough discussion on the problem of obtaining the nonvanishing proportion of $L$-functions at the central point from the $1$-level density.

Theorem \ref{1level} yields
\begin{equation}\label{density}
\frac{1}{|\mathcal{H}_{2g+1}|}\sum_{D\in\mathcal{H}_{2g+1}}\sum_{j=1}^{2g}\Phi(2g \theta_j)=\int_{-\infty}^{\infty}\widehat{\Phi}(y)\widehat{W}_{Sp(2g)}(y)dy+o(1)
\end{equation}
as $g\rightarrow\infty$ for any fixed, even $\Phi$ with the support of $\widehat{\Phi}$ in $(-2,2)$, where
\[
\widehat{W}_{Sp(2g)}(y)=\delta_0(y)-\tfrac{1}{2}\eta(y)
\]
with $\eta(y)$ being the characteristic function of the interval $[-1,1]$. Note that $\frac{1}{2}\eta(y)$ is the Fourier transform of $\frac{\sin2\pi x}{2\pi x}$.

Let
\[
p_m(g)=\frac{1}{|\mathcal{H}_{2g+1}|}\Big|\big\{D\in\mathcal{H}_{2g+1}:\ \textrm{ord}_{s=\frac{1}{2}}L(\tfrac{1}{2},\chi_D)=m\big\}\Big|.
\]
The functional equation implies that the order of zero of $L(s,\chi_D)$ at $s=\frac12$ is always even, so $p_{2m+1}(g)=0$. Also, it is clear that
\begin{equation}\label{cor1}
\sum_{m=0}^{\infty}p_m(g)=1.
\end{equation}
By choosing the test function $\Phi(x)$ such that $\Phi(x)\geq0$, $\Phi(0)=1$ and the support of $\widehat{\Phi}$ is in $(-2,2)$, we derive from \eqref{density} that
\[
\sum_{m=1}^{\infty}mp_m(g)\leq \int_{-\infty}^{\infty}\widehat{\Phi}(y)\widehat{W}_{Sp(2g)}(y)dy+o(1)
\]
as $g\rightarrow\infty$. Combining with \eqref{cor1} and the fact that $p_{2m+1}(g)=0$ we get
\begin{equation}\label{cor2}
p_0(g)\geq 1-\frac{1}{2}\int_{-\infty}^{\infty}\widehat{\Phi}(y)\widehat{W}_{Sp(2g)}(y)dy+o(1)
\end{equation}
as $g\rightarrow\infty$.

The Fourier pair
\[
\Phi(x)=\Big(\frac{\sin2\pi x}{2\pi x}\Big)^2,\quad \widehat{\Phi}(y)=\frac{1}{2}\Big(1-\frac{|y|}{2}\Big)\quad\textrm{if} \ |y|<2
\]
leads to quite a good result. In this case we have $\widehat{\Phi}(0)=\frac12$,
\[
\int_{-\infty}^{\infty}\widehat{\Phi}(y)\widehat{W}_{Sp(2g)}(y)dy=\frac{1}{2}-\frac{1}{4}\int_{-1}^{1}\Big(1-\frac{|y|}{2}\Big)dy=\frac{1}{8},
\]
and hence \eqref{cor2} gives $p_0(g)\geq \frac{15}{16}+o(1)=0.9375+o(1)$. This is the proportion that \"Ozl\"uk and Snyder \cite{ozluksnyder} obtained for the nonvanishing of quadratic Dirichlet $L$-functions at the central point.

We can obtain a slightly better result by solving the involving optimisation problem. The problem is to determine 
\[
\inf_{\Phi}\int_{-\infty}^{\infty}\widehat{\Phi}(y)\widehat{W}_{Sp(2g)}(y)dy,
\]
where the infimum is taken over all $\Phi\in\textrm{L}^1(\mathbb{R})$ such that $\Phi(x)\geq 0$, $\Phi(0)=1$ and the support of $\widehat{\Phi}$ is in $(-2,2)$.

The admissible functions $\Phi$ have the form $\Phi(x)=|f(x)|^2$, where $f\in \textrm{L}^2(\mathbb{R})$ is an entire function of exponential type $1$. Then
\[
\widehat{\Phi}(y)=(h\star\breve{h})(y),
\]
where
\[
\breve{h}(y)=\overline{h(-y)},\quad\textrm{support}(h)\subset[-1,1],\quad h\in\textrm{L}^2[-1,1].
\]
Under this presentation, the condition $\Phi(0)=1$ becomes $|\langle h,1\rangle|=1$. We define the self-adjoint operator $\textrm{K}_1:\textrm{L}^2[-1,1]\rightarrow\textrm{L}^2[-1,1]$ by
\[
\textrm{K}_1h(y)=(\tfrac{1}{2}\eta\star h)(y).
\]
Hence our problem is equivalent to minimising the quadratic form $$\langle (\textrm{I}+\textrm{K}_1)h,h\rangle$$ subject to the linear constraint $|\langle h,1\rangle|=1$.

As in Proposition A.1 \cite{ILS},
\[
\inf_{\substack{h\in\textrm{L}^2[-1,1]\\|\langle h,1\rangle|=1}}\langle (\textrm{I}+\textrm{K}_1)h,h\rangle=\frac{1}{\langle 1,h_0\rangle},
\]
where $h_0$ satisfies
\begin{equation}\label{cor3}
(\textrm{I}+\textrm{K}_1)h_0=1.
\end{equation}
In our case, \eqref{cor3} becomes
\[
1=h_0(y)+\int_{-1}^{1}\tfrac{1}{2}\eta(y-x)h_0(x)dx=h_0(y)-\frac{1}{2}\int_{0}^{1}h_0(x)dx-\frac{1}{2}\int_{0}^{1-y}h_0(x)dx
\]
for $0\leq y\leq 1$, and $h_0$ is an even function. Trying trigonometric functions we find that the even extension to $[-1,1]$ of 
\[
h_0(y)=\frac{\sin(\frac{y}{2}-\frac{\pi+1}{4})}{\sqrt{2}\sin\frac{1}{4}-\cos\frac{\pi+1}{4}}, \qquad 0\leq y\leq 1,
\]
solves the above equation. So,
\[
\inf_{\Phi}\int_{-\infty}^{\infty}\widehat{\Phi}(y)\widehat{W}_{Sp(2g)}(y)dy=\frac{1}{\langle 1,h_0\rangle}=\frac{\cot\frac{1}{4}-3}{8}.
\]
Hence \eqref{cor2} leads to
\[
p_0(g)\geq \frac{19-\cot(\frac{1}{4})}{16}+o(1)=0.9427\ldots+o(1)
\]
as $g\rightarrow\infty$, which finishes the proof of the corollary.

\section{Proof of Corollary \ref{sz}}
\label{simplezeros}
From Theorem \ref{paircorth} we get
\begin{equation}\label{correlation}
\frac{1}{2g|\mathcal{H}_{2g+1}|}\sum_{D\in\mathcal{H}_{2g+1}}\sum_{1\leq j, k\leq2g}\Phi\big(2g(\theta_{j,D}-\theta_{k,D})\big)=\int_{-\infty}^{\infty}\widehat{\Phi}(y)\big(\delta_0(y)+\eta(y)y\big)dy+o(1)
\end{equation}
as $g\rightarrow\infty$ for any fixed, even $\Phi$ with the support of $\widehat{\Phi}$ in $(-1,1)$. Here, as in the previous section, $\eta(y)$ is the characteristic function of the interval $[-1,1]$.

For $D\in\mathcal{H}_{2g+1}$, let $m_{j,D}$ be the multiplicity of the zero $q^{-\frac12}e^{2\pi i\theta_{j,D}}$ of $\mathcal{L}(u,\chi_D)$. Clearly,
\begin{equation}\label{cor2.1}
\sum_{\theta_{j,D}}{\!}^{*}\ m_{j,D}=2g,
\end{equation}
where $\sum^{*}$ denotes the summation over the distinct zeros among $\{\theta_{j,D}\}_{j=1}^{2g}$. Also, by choosing the test function $\Phi(x)$ such that $\Phi(x)\geq0$, $\Phi(0)=1$ and the support of $\widehat{\Phi}$ is in $(-1,1)$, we have
\[
\sum_{1\leq j, k\leq2g}\Phi\big(2g(\theta_{j,D}-\theta_{k,D})\big)\geq \sum_{\theta_{j,D}}{\!}^{*}\ m_{j,D}^2\geq 2\sum_{\theta_{j,D}}{\!}^{*}\ m_{j,D}-\sum_{\substack{\theta_{j,D}\\m_{j,D}=1}}{\!\!\!\!}^{*}\ \ 1.
\]
Combining with \eqref{correlation} and \eqref{cor2.1} we obtain
\begin{equation}\label{cor2.2}
\Bigg\langle \frac{1}{2g}\sum_{\substack{1\leq j\leq 2g\\ \theta_{j,D}\ \textrm{simple}}}1\Bigg\rangle_{\mathcal{H}_{2g+1}}\geq 2-\int_{-\infty}^{\infty}\widehat{\Phi}(y)\big(\delta_0(y)+\eta(y)y\big)dy+o(1)
\end{equation}
as $g\rightarrow\infty$. The problem is then reduced to determine 
\[
\inf_{\Phi}\int_{-\infty}^{\infty}\widehat{\Phi}(y)\big(\delta_0(y)+\eta(y)y\big)dy
\]
over all $\Phi\in\textrm{L}^1(\mathbb{R})$ such that $\Phi(x)\geq 0$, $\Phi(0)=1$ and the support of $\widehat{\Phi}$ is in $(-1,1)$.

The extremal problem we have is exactly like the one in the case of the Riemann zeta-function. This has been solved by Montgomery \cite{montgomery}, and that leads to
\[
\Bigg\langle \frac{1}{2g}\sum_{\substack{1\leq j\leq 2g\\ \theta_{j,D}\ \textrm{simple}}}1\Bigg\rangle_{\mathcal{H}_{2g+1}}\geq \frac{3}{2}-\frac{\cot(\frac{1}{\sqrt{2}})}{\sqrt{2}}+o(1)=0.6725\ldots+o(1)
\]
as $g\rightarrow\infty$.

Note that the usual Fourier pair
\[
\Phi(x)=\Big(\frac{\sin\pi x}{\pi x}\Big)^2,\quad \widehat{\Phi}(y)=1-|y|\quad\textrm{if} \ |y|<1,
\]
leads to the proportion of $\frac 23$.

\section{Appendix: Ratios Conjecture Calculations}
\label{app}

The calculations in this section are deeply heuristic. Interested readers unfamiliar with the ``recipe'' that leads to the following Conjecture \ref{RCF} are referred to, for example, \cite{conreysnaith}; Sections 2 and 3, for details.

\subsection{Ratios conjecture for $L$-functions in the hyperelliptic ensemble}

We would like to study
\[
R_g(\alpha;\beta):=\frac{1}{|\mathcal{H}_{2g+1}|}\sum_{D\in\mathcal{H}_{2g+1}}\frac{L(\frac{1}{2}+\alpha,\chi_D)}{L(\frac{1}{2}+\beta,\chi_D)}
\]
using the recipe in \cite{conreysnaith}. The shifts are assumed to satisfy the following conditions:
\begin{eqnarray*}
\big|\textrm{Re}(\alpha)\big| < \frac{1}{4},\quad q^{-(2g+1)}\ll \textrm{Re}(\beta)< \frac{1}{4}\quad\textrm{and}\quad
\textrm{Im}(\alpha), \textrm{Im}(\beta)\ll_\varepsilon q^{2g(g+1)(1-\varepsilon)}.\nonumber
\end{eqnarray*}
We use the approximate functional equation for the $L$-function in the numerator,
\[
L(\tfrac{1}{2}+\alpha,\chi_D)=\sum_{f\in\mathcal{M}_{\leq g}}\frac{\chi_D(f)}{|f|^{\frac12+\alpha}}+\mathcal{X}_D(\tfrac{1}{2}+\alpha)\sum_{f\in\mathcal{M}_{\leq g-1}}\frac{\chi_D(f)}{|f|^{\frac12-\alpha}},
\]
where
\[
\mathcal{X}_D(s)=\Big(\frac{|D|}{q}\Big)^{\frac12-s}=q^{2g(\frac12-s)},
\]
and the normal Dirichlet series expansion for that in the denominator,
\[
L(\tfrac{1}{2}+\beta,\chi_D)^{-1}=\sum_{h\in\mathcal{M}}\frac{\mu(h)}{|h|^{\frac12+\beta}}.
\]
The terms from the first part of the approximate functional equation contribute
\[
\frac{1}{|\mathcal{H}_{2g+1}|}\sum_{f,h}\frac{\mu(h)}{|f|^{\frac12+\alpha}|h|^{\frac12+\beta}}\sum_{D\in\mathcal{H}_{2g+1}}\chi_D(fh).
\]
We only retain the terms where $fh=\square$. Then from Lemma 3.7 the above expression is approximated by
\[
\sum_{fh=\square}\frac{\mu(h)a(fh)}{|f|^{\frac12+\alpha}|h|^{\frac12+\beta}},
\]
where
\[
a(fh)=\prod_{\substack{P\in\mathcal{P}\\P|fh}}\bigg(1+\frac{1}{|P|}\bigg)^{-1}.
\]
Using the Euler product we obtain (to save variables we now write $P^f$ and $P^h$ for $f$ and $h$, respectively)
\begin{align*}
&\prod_{P\in\mathcal{P}}\sum_{\substack{f,h\\f+h\ \textrm{even}}}\frac{\mu(P^h)a(P^{f+h})}{|P|^{(\frac12+\alpha)f+(\frac12+\beta)h}}\\
&\qquad\qquad=\prod_{P\in\mathcal{P}}\bigg(\sum_{f\ \textrm{even}}\frac{a(P^{f})}{|P|^{(\frac12+\alpha)f}}-\frac{|P|}{|P|+1}\sum_{f\ \textrm{odd}}\frac{1}{|P|^{(\frac12+\alpha)f+(\frac12+\beta)}}\bigg)\\
&\qquad\qquad=\zeta_q(1+2\alpha)\prod_{P\in\mathcal{P}}\bigg(1-\frac{1}{|P|^{1+2\alpha}}+\frac{|P|}{|P|+1}\frac{1}{|P|^{1+2\alpha}}-\frac{|P|}{|P|+1}\frac{1}{|P|^{1+\alpha+\beta}}\bigg)\\
&\qquad\qquad=\frac{\zeta_q(1+2\alpha)}{\zeta_q(1+\alpha+\beta)}A_g(\alpha;\beta),
\end{align*}
where
\begin{equation}\label{ARC}
A_g(\alpha;\beta)=\prod_{P\in\mathcal{P}}\bigg(1-\frac{1}{|P|^{1+\alpha+\beta}}\bigg)^{-1}\bigg(1-\frac{1}{|P|^{1+2\alpha}(|P|+1)}-\frac{1}{|P|^{\alpha+\beta}(|P|+1)}\bigg).
\end{equation}

The contribution of the terms coming from the second part of the approximate functional equation can be determined by using the functional equation
\[
L(\tfrac{1}{2}+\alpha,\chi_D)=q^{-2g\alpha }L(\tfrac{1}{2}-\alpha,\chi_D).
\]
Thus the recipe leads to the following ratios conjecture.

\begin{conjecture}\label{RCF}
We have
\begin{align*}
R(\alpha;\beta)=\frac{\zeta_q(1+2\alpha)}{\zeta_q(1+\alpha+\beta)}A_g(\alpha;\beta)+q^{-2g\alpha }\frac{\zeta_q(1-2\alpha)}{\zeta_q(1-\alpha+\beta)}A_g(-\alpha;\beta)+O_\varepsilon\big(q^{-g-\frac12+\varepsilon g}\big),
\end{align*}
where $A_g(\alpha;\beta)$ is defined as in \eqref{ARC}.
\end{conjecture}

We now take the derivative of the expression in Conjecture \ref{RCF} with respect to $\alpha$ and set $\alpha=\beta=r$. Note that $A_g(r;r)=1$,
\begin{equation}\label{A'}
A'_g(r;r)=\frac{\partial}{\partial \alpha}A_g(\alpha;\beta)\Big|_{\alpha=\beta=r}=\sum_{P\in\mathcal{P}}\frac{\log |P|}{(|P|^{1+2r}-1)(|P|+1)}
\end{equation}
and
\[
\frac{\partial}{\partial \alpha}q^{-2g\alpha}\frac{\zeta_q(1-2\alpha)}{\zeta_q(1-\alpha+\beta)}A_g(-\alpha;\beta)\Big|_{\alpha=\beta=r}=-(\log q)q^{-2gr }\zeta_q(1-2r)A_g(-r;r).
\]
Hence Conjecture \ref{RCF} leads to

\begin{theorem}\label{RCd}
Assuming that Conjecture \ref{RCF} holds, then we have
\begin{align*}
\frac{1}{|\mathcal{H}_{2g+1}|}\sum_{D\in\mathcal{H}_{2g+1}}\frac{L'(\frac{1}{2}+r,\chi_D)}{L(\frac{1}{2}+r,\chi_D)}=&\frac{\zeta'_q(1+2r)}{\zeta_q(1+2r)}+A'_g(r;r)-(\log q)q^{-2gr}\zeta_q(1-2r)A_g(-r;r)\\
&\qquad\qquad+O_\varepsilon\big(q^{-g-\frac12+\varepsilon g}\big),
\end{align*}
where $A_g(\alpha;\gamma)$ is defined as in \eqref{ARC}, and $A'_g(r;r)$ is defined as in \eqref{A'}.
\end{theorem}

Recall that $\Phi(2g\theta)=\phi(\theta)$ and $\phi(\theta)=\sum_{|n|\leq N}\widehat{\phi}(n)e(n\theta)$ is a real and even trigonometric polynomial. We would like to use Conjecture \ref{RCF} to compute the $1$-level density
\[
\Big\langle \Sigma_1(\Phi,g,D)\Big\rangle_{\mathcal{H}_{2g+1}}=\frac{1}{|\mathcal{H}_{2g+1}|}\sum_{D\in\mathcal{H}_{2g+1}}\sum_{j=1}^{2g}\Phi(2g \theta_j).
\]

It is easier to work in the $s$-world. So we shall first transform our $1$-level density problem from the $u$-world to that. Note that $\mathcal{L}(u,\chi_D)$ have $2g$ zeros at $u=q^{-\frac12}e^{2\pi i\theta_j}$. Hence {\it vertically}, $L(s,\chi_D)$ have zeros at $\frac12+i\gamma$ with periodicity $\frac{2\pi}{\log q}$, where
\[
\gamma = \frac{2\pi \theta_j}{\log q},\qquad 1\leq j\leq 2g.
\]
Thus,
\[
\sum_{j=1}^{2g}\Phi(2g \theta_j)=\sum_{0\leq \gamma<\frac{2\pi}{\log q}}\Phi\Big(\gamma\frac{g\log q}{\pi}\Big).
\]

By Cauchy's theorem, up to a negligible error term, we have
\begin{align*}
\frac{1}{|\mathcal{H}_{2g+1}|}\sum_{D\in\mathcal{H}_{2g+1}}\sum_{j=1}^{2g}\Phi(2g \theta_j)=&\frac{1}{|\mathcal{H}_{2g+1}|}\sum_{D\in\mathcal{H}_{2g+1}}\\
&\qquad\frac{1}{2\pi i}\bigg(\int_{\mathcal{C}_1}-\int_{\mathcal{C}_2}\bigg)\frac{L'(s,\chi_D)}{L(s,\chi_D)}\Phi\Big(-i(s-\tfrac{1}{2})\frac{g\log q}{\pi}\Big)ds,
\end{align*}
where $\mathcal{C}_1$ and $\mathcal{C}_2$ are the segments from $\frac12\pm\varepsilon$ to $\frac12\pm\varepsilon+\frac{2\pi i}{\log q}$, respectively.

The integral along $\mathcal{C}_1$ is
\[
\frac{1}{|\mathcal{H}_{2g+1}|}\frac{1}{2\pi}\int_{0}^{\frac{2\pi}{\log q}}\Phi\Big((t-i\varepsilon)\frac{g\log q}{\pi}\Big)\sum_{D\in\mathcal{H}_{2g+1}}\frac{L'(\frac{1}{2}+\varepsilon+it,\chi_D)}{L(\frac{1}{2}+\varepsilon+it,\chi_D)}dt.
\]
By Conjecture \ref{RCd}, the sum over $D$ is
\begin{align*}
\bigg(&\frac{\zeta'_q(1+2r)}{\zeta_q(1+2r)}+A'_g(r;r)-(\log q)q^{-2gr}\zeta_q(1-2r)A_g(-r;r)\bigg)\bigg|_{r=\varepsilon+it}+O_\varepsilon\big(q^{-g-\frac12+\varepsilon g}\big).
\end{align*}
Since the integrand is regular at $r=0$, we can move the line of integration to $\varepsilon=0$ and obtain
\begin{align*}
\frac{1}{2\pi}\int_{0}^{\frac{2\pi}{\log q}}\Phi&\Big(\frac{tg\log q}{\pi}\Big)\bigg(\frac{\zeta'_q(1+2it)}{\zeta_q(1+2it)}+A'_g(it;it)-(\log q)q^{-2itg}\zeta_q(1-2it)A_g(-it;it)\bigg)dt\\
&+O_\varepsilon\big(q^{(-g-\frac12+\varepsilon g}\big).
\end{align*}

For the integral along $\mathcal{C}_2$, we change variables, letting $s\rightarrow 1-s$, and use the functional equation
\[
-\frac{L'(1-s,\chi_D)}{L(1-s,\chi_D)}=-2g\log q-\frac{L'(s,\chi_D)}{L(s,\chi_D)},
\]
The contribution of the term $L'/L$ is exactly as before as $\Phi$ is even. Hence we have
\begin{align*}
\frac{1}{|\mathcal{H}_{2g+1}|}\sum_{D\in\mathcal{H}_{2g+1}}\sum_{j=1}^{2g}\Phi(2g \theta_j)&=\frac{1}{\pi}\int_{0}^{\frac{2\pi}{\log q}}\Phi\Big(\frac{tg\log q}{\pi}\Big)\bigg(g\log q+\frac{\zeta'_q(1+2it)}{\zeta_q(1+2it)}+A'_g(it;it)\\
&\qquad-(\log q)q^{-2itg}\zeta_q(1-2it)A_g(-it;it)\bigg)dt+O_\varepsilon\big(q^{-g-\frac12+\varepsilon g}\big)\\
&=A_1+A_2+A_3+A_4+O_\varepsilon\big(q^{-g-\frac12+\varepsilon g}\big),
\end{align*}
say.

Note that $\widehat{\Phi}(\frac{n}{2g})=2g\int_{0}^{1}\Phi(2g\theta)e(-n\theta)d\theta$. After a change of variables we get
\[
\widehat{\Phi}\Big(\frac{n}{2g}\Big)=\frac{g\log q}{\pi}\int_{0}^{\frac{2\pi}{\log q}}\Phi\Big(\frac{tg\log q}{\pi}\Big)q^{-itn}dt.
\]
Hence $A_1=\widehat{\Phi}(0)$ and
\begin{align*}
A_2=-\frac{1}{g}\sum_{n\leq \frac N2}\widehat{\Phi}\Big(\frac{n}{g}\Big),
\end{align*}
as
\[
\frac{\zeta'_q(1+2it)}{\zeta_q(1+2it)}=-\frac{q^{-2it}\log q}{1-q^{-2it}}=-(\log q)\sum_{n\geq 1}q^{-2itn}.
\]
For $A_3$, from \eqref{A'} we have
\[
A'_g(it;it)=\sum_{P\in\mathcal{P}}\frac{\log |P|}{(|P|^{1+2it}-1)(|P|+1)}=\sum_{\substack{P\in\mathcal{P}\\r\geq1}}\frac{\log |P|}{|P|^{r+2itr}(|P|+1)}.
\]
So
\begin{align*}
A_3&=\sum_{\substack{P\in\mathcal{P}\\r\geq1}}\frac{\log |P|}{|P|^{r}(|P|+1)}\frac{1}{\pi}\int_{0}^{\frac{2\pi}{\log q}}\Phi\Big(\frac{tg\log q}{\pi}\Big)|P|^{-2itr}\\
&=\frac{1}{g}\sum_{n\leq \frac N2}\widehat{\Phi}\Big(\frac{n}{g}\Big)\sum_{\substack{P\in\mathcal{P}_{n/r}\\r\geq1}}\frac{d(P)}{|P|^{r}(|P|+1)}.
\end{align*}
Finally, simple calculations show that
\[
A_g(-it;it)=\frac{\zeta_q(2)}{\zeta_q(2-2it)}.
\]
So
\[
\zeta_q(1-2it)A_g(-it;it)=\frac{q(1-q^{2it-1})}{(q-1)(1-q^{2it})}=\frac{1}{(q-1)}-\sum_{n\geq 1}q^{-2itn}.
\]
Hence
\[
A_4=-\frac{\widehat{\Phi}(1)}{g(q-1)}+\frac{1}{g}\sum_{n=g+1}^{N/2}\widehat{\Phi}\Big(\frac{n}{g}\Big).
\]
We thus obtain

\begin{theorem}
Assume that Conjecture \ref{RCF} holds, then we have
\begin{align*}
\Big\langle \Sigma_1&(\Phi,g,D)\Big\rangle_{\mathcal{H}_{2g+1}}  =\widehat{\Phi}(0)-\frac{1}{g} \sum_{n\leq g}\widehat{\Phi}\Big(\frac{n}{g}\Big)+c(\Phi,g)-\frac{\widehat{\Phi}(1)}{g(q-1)}+O_\varepsilon\big(q^{-g-\frac12+\varepsilon g}\big)
\end{align*}
for any $N$, where
\[
c(\Phi,g)=\frac{1}{g}\sum_{n\leq \frac N2}\widehat{\Phi}\Big(\frac{n}{g}\Big)\sum_{\substack{P\in\mathcal{P}_{n/r}\\r\geq1}}\frac{d(P)}{|P|^{r}(|P|+1)}.
\]
\end{theorem}

\textbf{Acknowledgments.}
This work started when both authors were attending the ``Number Theory and Function Fields at the Crossroads'' workshop, held in January $2016$ at the University of Exeter. We would like to thank the Heilbronn Institute, the University of Exeter and Julio Andrade for the organization and their hospitality. 
We are also grateful to Daniel Fiorilli, Jon Keating, Micah Milinovich, Zeev Rudnick and Kannan Soundararajan for their helpful comments on the paper. 

\bibliography{Arxiv}
\bibliographystyle{plain}

\end{document}